\documentclass[letterpaper, 11pt,  reqno]{amsart}

\usepackage{amsmath,amssymb,amscd,amsthm,amsxtra, esint}

\usepackage[margin=1.2in,marginparwidth=1.5cm, marginparsep=0.5cm]{geometry}

\usepackage[mathscr]{euscript}

\usepackage[implicit=true]{hyperref}

\usepackage{dsfont}

\allowdisplaybreaks[2]

\usepackage{cancel}

\usepackage{color}

\definecolor{gr}{rgb}   {0.,   0.69,   0.23 }
\definecolor{bl}{rgb}   {0.,   0.5,   1. }
\definecolor{mg}{rgb}   {0.85,  0.,    0.85}
\definecolor{yl}{rgb}   {0.8,  0.7,   0.}
\definecolor{or}{rgb}  {0.7,0.2,0.2}

\newtheorem{theorem}{Theorem} [section]
\newtheorem{THM:main}{Theorem}
\newtheorem{lemma}[theorem]{Lemma}
\newtheorem{proposition}[theorem]{Proposition}
\newtheorem{remark}[theorem]{Remark}

\newtheorem{definition}[theorem]{Definition}
\newtheorem{corollary}[theorem]{Corollary}

\newcommand{\noi}{\noindent}
\newcommand{\Z}{\mathbb{Z}}
\newcommand{\R}{\mathbb{R}}

\newcommand{\T}{\mathbb{T}}

\newcommand{\dl}{\delta}

\newcommand{\eps}{\varepsilon}

\newcommand{\ld}{\lambda}

\newcommand{\s}{\sigma}

\newcommand{\wt}{\widetilde}

\renewcommand{\l}{\ell}

\newcommand{\les}{\lesssim}
\newcommand{\ges}{\gtrsim}

\newcommand{\pa}{\partial}

\newcommand{\M}{\mathcal{M}}

\newcommand{\N}{\mathbb{N}}

\usepackage{tikz}

\usetikzlibrary{shapes.misc}
\usetikzlibrary{shapes.symbols}
\usetikzlibrary{shapes.geometric}
\tikzset{
	dot/.style={circle,fill=black,draw=black,inner sep=0pt,minimum size=0.5mm},
	>=stealth,
	}
\tikzset{
	ddot/.style={circle,fill=white,draw=black,inner sep=0pt,minimum size=0.8mm},
	>=stealth,
	}

\tikzset{decision/.style={ 
        draw,
        diamond,
        aspect=1.5
    }}

\tikzset{dia2/.style
={diamond,fill=white,draw=black,inner sep=0pt,minimum size=1mm},
	>=stealth,
	}

\tikzset{dia/.style
={star,fill=black,draw=black,inner sep=0pt,minimum size=1mm},
	>=stealth,
	}

\makeatletter
\def\DeclareSymbol#1#2#3{\expandafter\gdef\csname MH@symb@#1\endcsname{\tikz[baseline=#2,scale=0.15]{#3}}}
\def\<#1>{\csname MH@symb@#1\endcsname}
\makeatother

\DeclareSymbol{X}{-2.4}{\node[dot] {};}
\DeclareSymbol{1}{0}{\draw[white] (-.4,0) -- (.4,0); \draw (0,0)  -- (0,1.2) node[dot] {};}
\DeclareSymbol{2}{0}{\draw (-0.5,1.2) node[dot] {} -- (0,0) -- (0.5,1.2) node[dot] {};}

\DeclareSymbol{1}{-2.7}
 {
  \draw (0,0) node[dot]{};}

\DeclareSymbol{1'}{-2.7}
 {\draw (0,0) node[ddot]{};}

\DeclareSymbol{1''}{-2.7}
 {\draw (0,0) node[dia]{};}

\DeclareSymbol{3}{0}
 {\draw (0,0) node[dot]{} -- (0,1.2) node[ddot] {}; 
 \draw (-1.2,0) node[dot] {} -- (0,1.2)node[ddot] {} -- (1.2,0) node[dot] {};}

\DeclareSymbol{3'}{0}
 {\draw (0,0) node[dia]{} -- (0,1.2) node[ddot] {}; 
 \draw (-1.2,0) node[dia] {} -- (0,1.2)node[ddot] {} -- (1.2,0) node[dia] {};}

\DeclareSymbol{31}{-3}{
\draw (0,-1)node[dot] {} -- (0,0) node[ddot] {}-- (0.9, -1) node[dot] {}; 
\draw (0,-1)node[dot] {} -- (0,0) node[ddot] {}-- (-0.9, -1) node[dot] {}; 
\draw (0,0)node[ddot] {} -- (1.3,1) node[ddot] {}-- (1.3, 0) node[dot] {}; 
\draw  (1.3,1) node[ddot] {}-- (2.6, 0) node[dot] {}; 
}

\DeclareSymbol{32}{-3}{
\draw (0,-1)node[dot] {} -- (0,0) node[ddot] {}-- (0.9, -1) node[dot] {}; 
\draw (0,-1)node[dot] {} -- (0,0) node[ddot] {}-- (-0.9, -1) node[dot] {}; 
\draw (0,0)node[ddot] {} -- (0,1) node[ddot] {}-- (0.9, 0) node[dot] {}; 
\draw (0,0)node[ddot] {} -- (0,1) node[ddot] {}-- (-0.9, 0) node[dot] {}; 
}

\DeclareSymbol{33}{-3}{
\draw (2.6,-1)node[dot] {} -- (2.6,0) node[ddot] {}-- (3.5, -1) node[dot] {}; 
\draw (2.6,-1)node[dot] {} -- (2.6,0) node[ddot] {}-- (1.7, -1) node[dot] {}; 
\draw (0,0)node[dot] {} -- (1.3,1) node[ddot] {}-- (1.3, 0) node[dot] {}; 
\draw  (1.3,1) node[ddot] {}-- (2.6, 0) node[ddot] {}; 
}

\DeclareSymbol{3''}{0}
 {\draw (0,0) node[dia]{} -- (0,1.2) node[ddot] {}; 
 \draw (-1.2,0) node[dot] {} -- (0,1.2)node[ddot] {} -- (1.2,0) node[dia] {};}

\DeclareSymbol{3'''}{0}
 {\draw (0,0) node[dot]{} -- (0,1.2) node[ddot] {}; 
 \draw (-1.2,0) node[dot] {} -- (0,1.2)node[ddot] {} -- (1.2,0) node[dia] {};}

\DeclareSymbol{31'}{-3}{
\draw (0,-1.2)node[dia] {} -- (0,0) node[ddot] {}-- (0.9, -1.2) node[dia] {}; 
\draw (0,-1.2)node[dia] {} -- (0,0) node[ddot] {}-- (-0.9, -1.2) node[dia] {}; 
\draw (0,0)node[ddot] {} -- (1.3,1.2) node[ddot] {}-- (1.3, 0) node[dia] {}; 
\draw  (1.3,1.2) node[ddot] {}-- (2.6, 0) node[dia] {}; 
}

\DeclareSymbol{32'}{-3}{
\draw (0,-1.2)node[dia] {} -- (0,0) node[ddot] {}-- (0.9, -1.2) node[dia] {}; 
\draw (0,-1.2)node[dia] {} -- (0,0) node[ddot] {}-- (-0.9, -1.2) node[dia] {}; 
\draw (0,0)node[ddot] {} -- (0,1.2) node[ddot] {}-- (0.9, 0) node[dia] {}; 
\draw (0,0)node[ddot] {} -- (0,1.2) node[ddot] {}-- (-0.9, 0) node[dot] {}; 
}

\DeclareSymbol{33'}{-3}{
\draw (2.6,-1.2)node[dia] {} -- (2.6,0) node[ddot] {}-- (3.5, -1.2) node[dia] {}; 
\draw (2.6,-1.2)node[dia] {} -- (2.6,0) node[ddot] {}-- (1.7, -1.2) node[dia] {}; 
\draw (0,0)node[dot] {} -- (1.3,1.2) node[ddot] {}-- (1.3, 0) node[dot] {}; 
\draw  (1.3,1.2) node[ddot] {}-- (2.6, 0) node[ddot] {}; 
}

\numberwithin{equation}{section}
\numberwithin{theorem}{section}

\makeatletter
\@namedef{subjclassname@2020}{%
  \textup{2020} Mathematics Subject Classification}
\makeatother

\begin{document}
\baselineskip = 14pt

\title[HNLS on torus]{Well-posedness for the periodic Hyperbolic nonlinear Schr\"odinger equations}

\author[E.~Ba\c{s}ako\u{g}lu and Y.~Wang]
{Engin Ba\c{s}ako\u{g}lu and Yuzhao Wang}

\address{Engin Ba\c{s}ako\u{g}lu\\Institute of Mathematical Sciences\\ ShanghaiTech University\\ Shanghai\\ 201210\\ China}
\email{ebasakoglu@shanghaitech.edu.cn}

\address{Yuzhao Wang\\School of Mathematics\\ University of Birmingham\\ Watson Building\\ Edgbaston\\ Birmingham \\B15 2TT\\ United Kingdom}
\email{y.wang.14@bham.ac.uk}



\begin{abstract}
We establish local well-posedness for the hyperbolic nonlinear Schrödinger equation (HNLS) in the critical spaces.  
Following the approach of Killip–Vișan \cite{KV14}, we derive scale-invariant Strichartz estimates for HNLS on both rational and irrational tori, thereby removing the $\varepsilon$-loss of derivative present in the hyperbolic Strichartz estimates of Bourgain-Demeter \cite{BD17}.

\end{abstract}



%
\maketitle

\vspace{-3mm}

\tableofcontents

\section{Introduction}
We consider the (hyperbolic) nonlinear Schr\"odinger equations,
\begin{equation}
\label{eq:NLS}
\begin{cases}
i\pa_t u + \Delta_{\pm} u \pm|u|^{2m}u=0,\\
u(t,x)|_{t=0}=u_0(x),
\end{cases}
\quad (t,x)\in[0,T)\times\M,
\end{equation}
where $\M = \R^d$ and $\T^d$,
 $x=(x_1,...,x_{d}) \in \M$, and
\begin{align}\label{Dpm}
\Delta_{\pm} = \sum_{j=1}^{j_0} \eps_j  \partial_{x_j}^2 - \sum_{j={j_0}+1}^d \eps_j \partial_{x_j}^2
\end{align}
for some ${j_0} \in \{0,1,...,d\}$, and $\eps_j\in \mathbb{R}_+$ for ${j} \in\{1,...,d\}$. 
When ${j_0} = 0$, we write $\Delta_\pm = -\Delta$; when ${j_0} = d$, we write $\Delta_\pm = \Delta$, where $\Delta$ is the standard (elliptic) Laplacian operator given by
\[
    \Delta = \sum_{j=1}^d \eps_j \partial_{x_j}^2,
\]
with $\eps_j \in \R_+$.
When \( {j_0} \notin \{0, d\} \) in \eqref{Dpm}, 
we refer to \(\Delta_\pm\) as the hyperbolic Laplacian operator, 
and the corresponding Schrödinger equation \eqref{eq:NLS} as the hyperbolic nonlinear Schrödinger equation (HNLS),
which will be the main concern of this paper.
As a model, HNLS appears in the dynamics of deep-water gravity waves, \cite{AS79}; it also appears in the study of self-focusing wave packets in plasma or nonlinear optics; see, for instance, \cite{B98, BK96, BR96}. Over the last decades, hyperbolic nonlinear Schr\"odinger equations have been studied extensively in different settings; we refer to \cite{GT12,G13,YW13_2,MT15, Totz3, BD17, DMPS18, T20, BSTW25, BOW25, LZ}. 

\subsection{Strichartz estimates}
In this paper, we focus on the local well-posedness theory of \eqref{eq:NLS}.
For this purpose, we view the nonlinear problem \eqref{eq:NLS} as the nonlinear perturbation of the following linear equation 
\begin{align}\label{linearNLS}
\begin{cases}
i\pa_t u + \Delta_{\pm} u=0,\\  u(t)|_{t=0}=u_0,  
\end{cases}
\end{align}
whose solution is of the form
\begin{align*}
u(t,x)=e^{it\Delta_{\pm}}u_0 (x),
\end{align*}
where $e^{it\Delta_{\pm}}$ is the linear propagator that will be made explicit later.  
Therefore, in the study of \eqref{eq:NLS}, it is crucial to estimate solutions to \eqref{linearNLS} for given initial data $u_0$.
From the Plancherel identity, we see that $e^{it\Delta_{\pm}}$ is unitary in the Sobolev space, 
\begin{align*}
\| e^{it \Delta_{\pm}} f\|_{L_x^2 (\M)} = \|f \|_{L_x^2 (\M)}.
\end{align*}
Strichartz estimates of the following form are of particular importance:
\begin{align}\label{Stri}
\| e^{it \Delta_{\pm}} f \|_{L^q_t L_x^r (\R \times \M)} \les \| f\|_{H_x^s (\M)},
\end{align}
where the space-time norm is defined by
\[
\|u \|_{L^q_t L_x^r (\R \times \M)} = \big\| \| u (t,x)\|_{L_x^r (\M)}\big\|_{L_t^q (\R)},
\]
and we write $L^q_t L_x^r$ as $L^r_{t,x}$ when $q=r$.
It is known that estimates \eqref{Stri} crucially depend on the hyperbolic operator $\Delta_\pm$ as well as the underlying manifold $\M$.

When \(\mathcal{M} = \mathbb{R}^d\), it is shown in \cite{KT98} that
\begin{align}\label{Stri_R}
    \| e^{it \Delta_{\pm}} f \|_{L^q_t L^r_x (\mathbb{R} \times \mathbb{R}^d)} \lesssim \| f \|_{L^2_x (\mathbb{R}^d)},
\end{align}
if and only if \((q, r)\) satisfies
\[
    \frac{2}{q} + \frac{d}{r} = \frac{d}{2}, \quad 2 \leq q, r \leq \infty,
\]
and \((q, r, d) \neq (2, \infty, 2)\).
Note that the constant in \eqref{Stri_R} may depend on \(\varepsilon_j \in \mathbb{R}_+\) for \(j \in \{1, 2, \dots, d\}\), but it is uniform in the sign of \(\Delta_\pm\), i.e., independent of the choice of \({j_0} \in \{0, 1, 2, \dots, d\}\) in \eqref{Dpm}. In particular, the linear elliptic and hyperbolic Schrödinger operators share the same Strichartz estimate \eqref{Stri_R}.
This uniformity follows from the dispersive estimate
\begin{align}\label{disp1}
    \| e^{it \Delta_{\pm}} f \|_{L^\infty_x (\mathbb{R}^d)} \lesssim |t|^{-\frac{d}{2}} \| f \|_{L^1_x (\mathbb{R}^d)},
\end{align}
which is independent of \({j_0}\) in \eqref{Dpm}.

We turn to the case $\M = \T^d$.
The solution to \eqref{linearNLS}
have the form
\begin{align}
\label{linearSol}
e^{it\Delta_{\pm}}u_0 (x)
 = \sum_{k\in\mathbb{Z}^d} e^{2\pi i ( x\cdot k - t |k|_{\pm}^2 )}\widehat{u}_0 (k)
\end{align}
where $|k|_\pm$ is the symbol of the (hyperbolic) Laplacian operator $\Delta_\pm$ in \eqref{Dpm} given by
\begin{align}\label{kpm}
|k|_{\pm}^2 = \sum_{j=1}^{j_0} \eps_j  {k_j}^2 - \sum_{j={j_0}+1}^d \eps_j {k_j}^2,
\end{align}
for $\eps_j \in \R_+$.
Compatible with the notation for $\Delta_{\pm}$, we set 
\begin{align}\label{d(j_0)}
  \delta (j_0)=\min\{j_0, d-j_0\}.  
\end{align}
It is known that there is no decay estimate like \eqref{disp1} for \eqref{linearSol}; see Lemma \ref{LEM:disK}.
Strichartz estimates for periodic Schr\"odinger operator \eqref{linearSol} are much more involved than that of the Eucleadean ones in \eqref{Stri_R} and are sensitive the the choice of ${j_0}$ in \eqref{kpm},
which are summarised in the following.

\begin{theorem}
[Theorem 2.4 in \cite{BD15} and Corollary 1.3 in \cite{BD17}]
\label{THM:Stri_T}
Fix $d\geq 1$, $\eps_1,...,\eps_d\in (0,1]$, 
$1\leq N\in 2^{\mathbb{Z}}$, 
and $p\geq \frac{2(d+2-\delta (j_0))}{d-\delta (j_0)}$ for $j_0\in\{0, 1,...,d\}$. 
Then, 
for each $\eps>0$, 
we have 
\begin{align}\label{BDstrcest1}
\Vert e^{it\Delta_{\pm}}P_{\leq N}\phi\Vert_{L^p_{t,x}([0,1]\times\mathbb{T}^d)}\lesssim_{\eps} N^\eps K_{p,d,j_0} (N)
\Vert P_{\leq N}\phi\Vert_{L^2(\T^d)},
\end{align}
where $e^{it\Delta_{\pm}} u_0 (x)$ is given in \eqref{linearSol} and \eqref{kpm} with $j_0 \in \{0,1,\cdots,d\}$, and 
\begin{align*}
K_{p,d,j_0} (N) = 
\begin{cases}
N^{\frac{d}2 - \frac{d+2}p}, & p \ge \frac{2(d+2-\delta (j_0))}{d-\delta (j_0)},\\
N^{(\frac12 - \frac1p) \delta (j_0)}, & 2 \le p < \frac{2(d+2-\delta (j_0))}{d-\delta (j_0)},
\end{cases}
\end{align*}
where $\delta (j_0)$ is as in \eqref{d(j_0)}.
\end{theorem}

\begin{remark}\rm
When $j_0 \in \{0,d\}$, i.e., $\delta (j_0) = 0$,
Theorem \ref{THM:Stri_T} was proved in \cite[Theorem 2.4]{BD15} as a consequence of the $\l^2$ decoupling theorem for parabola.
When $j_0 \in \{1,\cdots,d-1\}$, Theorem \ref{THM:Stri_T} was much more involved.
When $d = 2$ and $j_0 = 1$, \eqref{BDstrcest1} was first proved by the second author \cite{YW13_2} via a lattice counting argument,
which was then extended to the general case in \cite[Corollary 1.3]{BD17} as a consequence of the $\l^2$ decoupling for hypersurfaces with nonzero Gaussian curvature.
\end{remark}

We remark that the estimates \eqref{BDstrcest1} are insufficient for proving the well-posedness of the nonlinear problem \eqref{eq:NLS} due to the derivative loss \( N^\varepsilon \).  
For the elliptic case (i.e., \( j_0 \in \{0, d\} \)), Bourgain \cite{B93} removed this derivative loss on the rational torus for \( p > 4 \). This result was later generalized to the general torus by \cite{OhWang} and further extended to the wider range \( p > \frac{2(d+2)}{d} \) by \cite{KV14}. 

\begin{theorem}
[Theorem 2 in \cite{KV14}]
\label{THM:Stri_scale}
Fix $d\geq 1$, $\eps_1,...,\eps_d\in (0,1]$, 
$1\leq N\in 2^{\mathbb{Z}}$, 
and $p > \frac{2(d+2)}{d}$. 
Then, 
we have 
\begin{align*}
\Vert e^{\pm it\Delta} P_{\leq N}\phi\Vert_{L^p_{t,x}([0,1]\times\mathbb{T}^d)}\lesssim  N^{\frac{d}{2}-\frac{d+2}{p}}
\Vert P_{\leq N}\phi\Vert_{L^2(\mathbb{T}^d)},
\end{align*}
where the implicit constant is independent of $N$, and $e^{\pm it\Delta} u_0 (x)$ is given in \eqref{linearSol} and \eqref{kpm} with $j_0 \in \{0,d\}$. 
\end{theorem}

For the hyperbolic case (i.e., \( j_0 \in \{1, \dots, d-1\} \)), the second author \cite{YW13_2} eliminated the derivative loss when \( d = 2 \), \( j_1 = 1 \), and \( 2 \leq p \le 4 \) on the rational torus.  
In the general hyperbolic case, Bourgain and Demeter \cite{BD17} conjectured that the loss \( N^\varepsilon \) in \eqref{BDstrcest1} could be removed for \( p > \frac{2(d+2 - \delta (j_0))}{d - \mathscr \delta (j_0)} \). 
We confirm this in the following.

\begin{theorem}
\label{THM:scale2}
Fix $d\geq 2$, $\eps_1,...,\eps_d\in (0,1]$, 
$1\leq N\in 2^{\mathbb{Z}}$, 
and $p > \frac{2(d+2-\delta (j_0))}{d-\delta (j_0)}$ for $j_0\in\{0, 1,...,d\}$. 
Then,   
we have 
\begin{align}\label{BDstrcest2}
\Vert e^{it\Delta_{\pm}}P_{\leq N}\phi\Vert_{L^p_{t,x}([0,1]\times\mathbb{T}^d)}\lesssim 
N^{\frac{d}{2}-\frac{d+2}{p}}
\Vert P_{\leq N}\phi\Vert_{L^2(\T^d)},
\end{align}
where $e^{it\Delta_{\pm}} u_0 (x)$ is given in \eqref{linearSol} and \eqref{kpm} with $j_0 \in \{0,1,\cdots,d\}$.
\end{theorem}

The proof of Theorem \ref{THM:scale2} for $\delta (j_0) = 0$ has been proved in Theorem \ref{THM:Stri_scale}; while for $\delta (j_0) \neq 0$, Theorem \ref{THM:scale2} will be presented in Section~\ref{epsilonremoval}.

\medskip

\begin{remark}\label{RMK:c1}
\rm 
The derivative loss of order $N^\varepsilon$ in~\eqref{BDstrcest1} when 
\begin{equation}
p = \frac{2(d+2 - \delta (j_0))}{d - \delta (j_0)},
\end{equation}
is conjectured to be necessary for $d (j_0) = 0$, but not for $d (j_0) \neq 0$.
In the elliptic case where $\delta (j_0) = 0$, Bourgain~\cite{B93} first established this conjecture for $d = 1$, while Takaoka--Tzvetkov~\cite{TT01} confirmed it for $d = 2$. Recently, the sharp derivative loss of $(\log N)^{1/4}$ was obtained for $d = 2$ in~\cite{HK24}. For higher dimensions $d \geq 3$, the conjecture remains completely open.
For hyperbolic cases (general $\delta (j_0) \neq 0$), the problem is largely unresolved\footnote{See Remark \ref{RMK:final} for a recent updates.}. 
\end{remark}

\begin{remark} \label{RMK:c2}
\rm
A related conjecture suggests that the $N^\varepsilon$ loss in~\eqref{BDstrcest1} can be eliminated when 
\begin{equation}
2 < p < \frac{2(d + 2 - \delta (j_0))}{d - \delta (j_0)}.
\end{equation}
Partial results exist for this conjecture: For $d = 1$, the work~\cite{Zg} established the range $2 < p \leq 4$, while for $d = 2$ with $\delta (j_0) = 1$,~\cite{YW13_2} verified the same range. 
The problem remains open\footnote{See Remark \ref{RMK:final} for a recent updates on $d = 3$.} for dimensions $d \geq 4$ and general values of $\delta (j_0)$.
\end{remark}

\subsection{Well-posedness}
When $\M = \R^d$,
HNLS \eqref{eq:NLS} enjoys the scaling symmetry,
i.e., if $u$ solves \eqref{eq:NLS} with an initial data $u_0$,
then the rescaled function $u_{\ld} (t,x): = \ld^{\frac1m} u(\ld x, \ld^2 t)$ is also a solution to \eqref{eq:NLS} with the rescaled initial data $u_{0,\ld} (x): = \ld^{\frac1m} u(\ld x)$.
We say that the Sobolev index $s_c : = s_c (d,m)$ is critical if the homogeneous norm $\| \cdot\|_{\dot H^{s_c} (\R^d)}$ is invariant under the same scaling.
Particularly, the critical Sobolev index is
\begin{align}
\label{critical}
s_c = s_c (d,m) =\frac{d}2 - \frac1m.
\end{align} 
When $\M = \T^d$,
we do not have this scaling invariance.
However, the critical Sobolev index \eqref{critical} still provides important heuristics.

When $d (j_0) = 0$, i.e., the elliptic Schr\"odinger equation, the local well-posedness for \eqref{eq:NLS} in critical/subcritical spaces have been well-studied.
See \cite{B93,BD15,HTT11,YW13_2,KV14}.

\begin{theorem}
[elliptic case]
\label{THM:elliptic}
Let $\delta (j_0) = 0$, and let $s_{c}$ be as defined in \eqref{critical}. For integers $d \ge 2$ and $m \ge 1$, the Cauchy problem \eqref{eq:NLS} satisfies the following properties: It is locally well-posed in $H^{s}(\mathbb{T}^d)$ for all $s > s_c$. When $(d, m) \neq (2,1)$, the local well-posedness extends to the critical space $H^{s_c}(\mathbb{T}^d)$.\footnote{In the case $(d,m) = (3,1)$, the well-posedness of \eqref{eq:NLS} in the critical space $H^{\frac12}(\T^3)$ was left open in \cite{YW13}. 
The proof follows the same strategy as in other cases treated in \cite{YW13}, with the additional use of Strichartz estimates from \cite{BD15,KV14}. 
For completeness, we provide the proof in the appendix.
}
\end{theorem}

The subcritical results in Theorem~\ref{THM:elliptic}, i.e., the well-posedness of \eqref{eq:NLS} in $H^s(\mathbb{T}^d)$ for $s > s_c$, were established in \cite{B93}. 
The critical case ($s = s_c$) of Theorem~\ref{THM:elliptic} is more delicate. Herr--Tataru--Tzvetkov \cite{HTT11} first proved the well-posedness of \eqref{eq:NLS} for $(d,m) = (3,2)$ in the critical space $H^1(\mathbb{T}^3)$, developing a novel multilinear argument based on temporal orthogonality (see Remark~\ref{RMK:failH} for details). 
The second author \cite{YW13_2} extended these critical results to other cases, establishing Theorem~\ref{THM:elliptic} except for $(d,m) = (3,1)$ and $(d,m) = (4,1)$. These remaining cases were resolved by Killip--Visan \cite{KV14} using scale-invariant Strichartz estimates (Theorem~\ref{THM:Stri_scale}) and a bilinear approach. 
We note that the bilinear method in \cite{KV14} differs from the temporal orthogonality argument in \cite{HTT11} and appears more suitable for the hyperbolic setting considered in this paper. 
See Remark~\ref{RMK:failH} for further discussions.

For the case $(d,m) = (2,1)$, the critical exponent is $s_c = s_c(2,1) = 0$, making $L^2(\mathbb{T}^2)$ the critical space. 
Kishimoto \cite{Kishimoto} proved that the solution map of \eqref{eq:NLS} fails to be $C^3$ in $L^2(\mathbb{T}^2)$. 
For super-critical regularities $s < s_c$, norm inflation phenomena for solutions to \eqref{eq:NLS} in $H^s$ were established in \cite{Kishi2}. 
We say norm inflation occurs in $H^s$ if for every $\varepsilon > 0$, there exist initial data $\phi \in H^\infty$ and a time $T > 0$ satisfying
\[
\|\phi\|_{H^s} < \varepsilon \quad \text{and} \quad 0 < T < \varepsilon,
\]
such that the corresponding smooth solution $u$ exists on $[0,T]$ and satisfies
\[
\|u(T)\|_{H^s} > \varepsilon^{-1}.
\]
This demonstrates that Theorem~\ref{THM:elliptic} is sharp regarding analytic well-posedness.
If we consider weaker notions of well-posedness requiring only continuity of the solution map, then $(d,m) = (2,1)$ remains the sole unresolved case. 
The question of $L^2(\mathbb{T}^2)$ well-posedness with merely continuous solution map remains a major open problem in the field.

Recently, \cite{KW24} investigated general non-algebraic nonlinearities, in contrast to the integer-power nonlinearities ($m \in \mathbb{Z}_+$) treated in Theorem~\ref{THM:elliptic}. However, our current work focuses exclusively on algebraic nonlinearities.

\medskip

One of the main purpose of this paper is to extend Theorem \ref{THM:elliptic} to general hyperbolic setting \eqref{eq:NLS}.
With the help of Theorem \ref{THM:Stri_T} and similar arguments as in \cite{B93}, we can establish the sub-critical results as follows.

\begin{theorem}[hyperbolic case: subcritical]
\label{THM:main0}
Let $s_{c}$ be as in \eqref{critical} and $d \ge 2, m \ge 1,  j_0 \ge 1$ be positive integers. 
Then, the Cauchy problem \eqref{eq:NLS} is locally well-posed in $H^{s} (\T^d)$ for $s > s_c$ provided: 
\begin{itemize}
\item $d=2$ and $m\geq2$.
\item $d\ge 3$ and $m\geq1$. 
\end{itemize} 
\end{theorem}

\begin{remark}\rm
For $d = 2$ and $m = 1$, the sharp local well-posedness of \eqref{eq:NLS} has been established in \cite{B93} and \cite{YW13_2}.  
In particular, when $\delta (j_0) = 0$, \eqref{eq:NLS} reduces to the standard nonlinear Schrödinger equation, which is:  
well-posed in $H^s(\mathbb{T}^2)$ for $s > 0$ (see \cite{B93}); and ill-posed in $H^s(\mathbb{T}^2)$ for $s < 0$ (see \cite{TT01} and \cite{Kishimoto}).  
When $j_0 = 1$, the Strichartz estimates deteriorate, leading to a weaker well-posedness theory. The second author \cite{YW13_2} proved that \eqref{eq:NLS} is well-posed in $H^s(\mathbb{T}^2)$ for $s > \frac{1}{2}$;  
and is (mild) ill-posed in $H^s(\mathbb{T}^2)$ for $s < \frac{1}{2}$.\footnote{Recently, \cite{LZ} extended the ill-posedness result to $H^{\tfrac12}(\T^2)$.}

For positive integers $m \ge 1$, Theorem \ref{THM:main0}, together with \cite{B93} and \cite{YW13_2}, provides a complete characterization of well-posedness and ill-posedness for \eqref{eq:NLS} up to the critical regularity.  
\end{remark}

\begin{remark}\rm
In a recent paper \cite{SW24}, Saut and Wang thoroughly reviewed hyperbolic nonlinear Schr\"odinger equations, and proposed a set of conjectures and open problems. In this regard, we point out that the result of Theorem \ref{THM:main0} and \ref{THM:main} in the case $d=3$ provides an answer in relation to the ``open problem $3$" proposed in \cite{SW24}. 
\end{remark}

Our next result addresses the well-posedness in the critical spaces $H^{s_c}$, where $s_c$ is the critical exponent given in \eqref{critical}.

\begin{theorem}[hyperbolic case: critical]
\label{THM:main}
Let $s_{c}$ be as in \eqref{critical} and $d \ge 2, m, j_0$ be positive integers with $\delta (j_0) \neq 0$. 
Then, the Cauchy problem \eqref{eq:NLS} is locally well-posed in $H^{s_c} (\T^d)$ provided: 
\begin{itemize}
\item $d=2$ and $m\geq4$.
\item $d=3, 4$ and $m\geq2$.\footnote{See Remark \ref{RMK:final} for a recent updates.}
\item $d=4$, $m=1$, and $\delta (j_0) \neq 2$.
\item $d\geq 5$ and $m\geq1$.
\end{itemize}
\end{theorem}

\begin{remark}\rm
\label{RMK:ad}
When $\delta (j_0) = 0$ (the elliptic case), Theorem~\ref{THM:elliptic} establishes well-posedness in critical spaces $H^{s_c}$ for a broad range of $(d, m)$ pairs. 
Theorem~\ref{THM:Stri_scale} suggests that whenever we have a sharp (scale-invariant) Strichartz estimate of the form
\[
\| e^{\pm it\Delta} P_{\leq N}\phi \|_{L^{2(m+1)}_{t,x}([0,1]\times\mathbb{T}^d)} \lesssim N^{\frac{d}{2}-\frac{d+2}{2(m+1)}} \| P_{\leq N}\phi \|_{L^2(\mathbb{T}^d)},
\]
we should expect a corresponding well-posedness theory for \eqref{eq:NLS} in $H^{s_c}(\mathbb{T}^d)$, where $s_c = s_c(d,m)$.
By analogy with the elliptic setting, we can ask whether this relationship persists in the hyperbolic case ($\delta (j_0) \neq 0$) - specifically, whether we can establish well-posedness for \eqref{eq:NLS} in $H^{s_c}(\mathbb{T}^d)$ for those $(d,m)$ pairs satisfying
\[
\| e^{it\Delta_{\pm}} P_{\leq N}\phi \|_{L^{2(m+1)}_{t,x}([0,1]\times\mathbb{T}^d)} \lesssim N^{\frac{d}{2}-\frac{d+2}{2(m+1)}} \| P_{\leq N}\phi \|_{L^2(\mathbb{T}^d)},
\]
where $\Delta_\pm$ denotes the hyperbolic Laplacian defined in \eqref{Dpm} with $\delta (j_0) \neq 0$.

Theorem~\ref{THM:main}, in light of Theorem~\ref{THM:scale2}, provides an affirmative answer to this question for all cases except $(d,m) = (2,3)$. While we conjecture that the result should hold for $(d,m) = (2,3)$ in the hyperbolic setting as well, establishing this would require techniques beyond those developed in this paper. We plan to address this remaining case in future work.
\end{remark}

The first critical well-posedness result for NLS on compact manifolds was established by Herr--Tataru--Tzvetkov \cite{HTT11} in the elliptic setting, with subsequent developments in \cite{YW13}. The key innovation in \cite{HTT11} involved multilinear estimates through a temporal orthogonality argument, which exploits additional decay properties from localized frequencies. 

To illustrate the temporal orthogonality approach for a quadratic nonlinearity (see \cite{HTT14} for complete details), let $\mathcal{R}_M(N)$ denote the collection of all subsets of $\mathbb{Z}^d$ formed by intersecting cubes of side length $2N$ with strips of width $2M$. For parameters $1 \leq N_2 \leq N_1$, consider $C = (\xi_0 + [-N_2,N_2]^n) \cap \mathbb{Z}^d$ with $|\xi_0| \sim N_1$. The essential task is to verify that the functions $\{P_{R_k}u_1u_2\}_k$ are almost orthogonal in $L^2(I \times \mathbb{T}^d )$, where
\begin{align*}
C = \bigcup_{k \in \mathbb{Z}: |k| \sim N_1/N_2} R_k, \quad R_k \in \mathcal{R}_M(N_2), \quad M := \max\{N_2^2/N_1, 1\}.
\end{align*}
The crucial observation comes from the behavior of $\xi_1 \in R_k$:
\begin{align*}
|\xi_1|^2 = \frac{1}{|\xi_0|^2}|\xi_1 \cdot \xi_0| + |\xi_1 - \xi_0|^2 - \frac{1}{|\xi_0|^2}|(\xi_1 - \xi_0) \cdot \xi_0|^2 = M^2k^2 + O(M^2k).
\end{align*}
This implies that the intervals
\begin{align}\label{setsofalmort}
[-M^2k^2 - cM^2|k|, -M^2k^2 + cM^2|k|]
\end{align}
are essentially disjoint, which yields the desired almost orthogonality in $L^2(I \times \mathbb{T}^d )$.

\begin{remark}\rm
\label{RMK:failH}
For the hyperbolic NLS (HNLS) \eqref{eq:NLS}, applying this approach leads to
\begin{equation}
\begin{aligned}\label{HNLSalmostorth}
|\xi_1|_-^2 &= (\xi_1 \cdot a)^2 + |\xi_1 - \xi_0|^2 - \big((\xi_1 - \xi_0) \cdot a\big)^2 - 2\sum_{j \leq k} (\xi_{1j} - \xi_{0j})^2 \\
&\quad - 4\sum_{j \leq j_0}(\xi_{1j} - \xi_{0j})\xi_{0j} - 2\sum_{j \leq j_0}\xi_{0j}^2,
\end{aligned}
\end{equation}
where $|\xi_1|^2_- = -\sum_{j \leq j_0}\xi_{1j}^2 + \sum_{j > j_0}\xi_{1j}^2$. However, the final two terms in \eqref{HNLSalmostorth} may disrupt the essential disjointness property of the associated sets (analogous to those in \eqref{setsofalmort}) when $|\xi_0| \sim N_1$. Consequently, our proof strategy for the nonlinear estimates in this paper relies exclusively on spatial frequency cube decompositions. This simplified approach proves sufficient to establish the critical results.
\end{remark}

Despite Remark \ref{RMK:failH}, there may exist other temporal decomposition that still preserve the orthogonality for the hyperbolic setting.
Here, instead, we follow a different approach in proving the critical well-posedness from \cite{KV14},
where the authors use bilinear estimates; and moreover,
there is no need to exploit the temporal orthogonality of free evolutions discussed in Remark \ref{RMK:failH}.

Notice that the case $j_0=d$ (which we ignore) corresponds to the elliptic Schr\"odinger equation. Therefore, with this notation, we introduce the notion of admissible pairs, which will be used in deciding exponents pertaining to the Strichartz estimates during the analysis of nonlinear estimates.

\begin{definition}\label{DEF:adm}
Let $j_0 \in\{0,1,2,...,d\}$. 
We say that $(d,p )\in \mathbb{N}\times \mathbb{R}$ is admissible pair if 
\begin{equation}\label{admissible}
\begin{aligned}
 p>\frac{2(d+2-\delta (j_0))}{d-\delta (j_0)},
\end{aligned}
\end{equation}
in view of Remark \ref{RMK:ad}.
\end{definition}
\begin{remark}
Since $\delta (j_0)=d(d-j_0)$, we observe that for $j_0\in\{1,2,...,d-1\}$, the following holds:
\begin{equation*}
\begin{aligned}
 \underbrace{\dl(1)= \dl(d-1)}_{=1}<...<\underbrace{\delta (j_0)=\dl(d-j_0)}_{=j_0}<...<\underbrace{\dl(d/2)}_{d/2},\quad\text{if}\,\,d\,\,\text{is}\,\,\text{even},
 \end{aligned}
\end{equation*}
and 
\begin{equation*}
\begin{aligned}
 \underbrace{\dl(1)=\dl(d-1)}_{=1}<...<\underbrace{\delta (j_0)=\dl(d-j_0)}_{=j_0}<...<\underbrace{\dl(\frac{d-1}{2})=\dl(\frac{d+1}{2})}_{\frac{d-1}{2}},\quad\text{if}\,\,d\,\,\text{is}\,\,\text{odd}.
\end{aligned}
\end{equation*}
 Therefore, for $d\geq 2$ and $j_0\in\{1,2,...,d-1\}$, the right hand side of \eqref{admissible} attains its largest value at $j_0=\frac{d}{2}$ if $d$ is even, and also at $j_0=\frac{d-1}{2}$ if $d$ is odd. By this means, the condition on $p$ with respect to Definition \ref{DEF:adm} covering all $j_0\in\{1,2,...,d-1\}$ can be read as follows
\begin{align}\label{worstcase}
p>\begin{cases}
   2+\frac{8}{d},\quad\quad\text{if}\,\,d\,\,\text{is even},\\ 2+\frac{8}{d+1},\quad\text{if}\,\,d\,\,\text{is odd}.
    \end{cases}
\end{align}
However, in the case $d=4$, $j_0=1,3$, $m=1$, the condition on $p$ given by \eqref{worstcase} is not useful in practice. Therefore, in view of Definition \ref{DEF:adm}, we choose an admissible pair $(4,p)$ with $p>\frac{10}{3}$ in this case. Moreover, once $j_0=d$ (purely elliptic case), we have $\dl(d)=0$, thus \eqref{admissible} implies that $p>\frac{2(d+2)}{d}$ for all $n\geq 1$. 
\end{remark}

\begin{remark}\label{RMK:final}
\rm 
After completing our work, 
we noticed that a paper was posted on arXiv today by \cite{LZ}, which claims to resolve the cases $d=3$ and $\delta (j_0)=1$, previously proposed as open problems in Remarks \ref{RMK:c1} and \ref{RMK:c2}.  
As a consequence, 
they \cite{LZ} establish the sharp (i.e., without $\varepsilon$-loss) Strichartz estimate in Theorem \ref{THM:scale2} for $d=3$ and $\delta (j_0)=1$, 
together with sharp local well-posedness in this setting, 
thereby obtaining the same results as Theorems \ref{THM:main0} and \ref{THM:main} but only for $d = 3$. 

We would like to emphasize that while \cite{LZ} obtained the sharp $L^4$ Strichartz estimate on $\T^3$, 
ours has an $\eps$-derivative loss.
Nevertheless, their well-posedness results coincide with ours in the case $d = 3$.

The proof of Strichartz estimates in \cite{LZ} is based on incidence geometry, similar to the approach in \cite{HK24} for the $L^4$ Strichartz estimate on $\T^2$. 
They then interpolate between the $L^4$ and $L^\infty$ estimates to obtain $L^p$ bounds for $p>4$. 
The proof of $L^p$ estimate for $p>4$ differs from our approach, 
which is instead based on an $\varepsilon$-removal argument and the $\l^2$-decoupling results from \cite{BD17}. 
Their $L^p$ estimates rely crucially on the sharp endpoint Strichartz estimate (in their case, the $L^4$ estimate when $(d,\delta (j_0))=(3,1)$).
By contrast, our method is more robust and applies to all other dimensions, where such endpoint Strichartz estimates are not available.
\end{remark}

\subsection{Outline of the paper} In Section \ref{preliminaries}, we introduce notations, function spaces and useful lemmas in order to use them in the rest of the paper. In Section \ref{scinvstr}, we follow the argument of Killip-Vişan to remove the $\eps$-loss from Theorem \ref{THM:Stri_T}. In Section \ref{SectNonlinear}, exploiting the result of Section \ref{epsilonremoval}, we prove nonlinear estimates using the critical function space theory $X^s$. Section \ref{wellposednessHNLS} is devoted to proving our main result (Theorem \ref{THM:main}). Finally, in Appendix \ref{appendix} we consider three dimensional cubic NLS equation and prove the related nonlinear estimate in the relevant critical function space.  

 \section{Preliminaries}\label{preliminaries}
 \subsection{Notation}
 We write $A\lesssim B$ to indicate that there is a constant $C>0$ such that $A\leq CB$, also denote $A\sim B$ when $A\lesssim B \lesssim A$. For the Fourier transform of functions, we follow the following conventions:
 \begin{align*}
     \widehat{f}(k)=\int_{\mathbb{T}^d } e^{-2\pi ik\cdot x}f(x)dx, 
    \quad f(x)=\sum_{k \in\mathbb{Z}^d}e^{2\pi ik\cdot x}\widehat{f}(k).
 \end{align*}
Let  $\varphi$ denote a smooth radial cutoff such that $\varphi(x)=1$ for $|x|\leq 1$ and $\varphi(x)=0$ for $|x|\geq 2$. We define the Littlewood-Paley projection operators 
\[
\widehat{P_1f}(k):=\widehat{f}(k)\prod_{j=1}^d \varphi(k_j),
\] 
where  $k=(k_1,...,k_d) \in \Z^d$.
For a dyadic number $N\geq 2$,
\begin{equation*}
\begin{aligned} 
\widehat{P_{\leq N}f}(k)&:=\widehat{f}(k)\prod_{j=1}^d\varphi(k_j/N),\\
\widehat{P_{N}f}(k)&:=\widehat{f}(k)\prod_{j=1}^d\big[\varphi(k_j/N)-\varphi(2k_j/N)\big].
\end{aligned}
\end{equation*}
More generally, for any measurable set $S\subset \mathbb{Z}^d$, we write $P_S$ to denote the Fourier projection operator with symbol $\chi_{S}$, where $\chi_{S}$ denotes the characteristic function of $S$.
\subsection{Function Spaces}
 In this part, we introduce the main function spaces $X^s$ and $Y^s$ which are used as substitutes for Fourier restriction spaces $X^{s,b}$ at the critical scaling index. These spaces are based on function spaces $U^p$ and $V^p$ and are initially used to construct solutions to dispersive PDES in \cite{HHK09, HTT11, HTT14}. In what follows, we discuss their definitions and properties for the hyperbolic Schr\"odinger equations, for a detailed discussion see \cite{HHK09, HTT11}. Assume $\mathcal{H}$ is a separable Hilbert space over $\mathbb{C}$, and $\mathcal{Z}$ denotes the set of finite partitions $0=t_0<t_1<...<t_k\leq T$ with the convention that $v(T):=0$ for all functions $v: [0,T)\rightarrow \mathcal{H}$.  
 \begin{definition}
 Let $1\leq p <\infty$. For $\{t_k\}_{k=0}^K\in \mathcal{Z}$ and $\{\phi_k\}_{k=0}^{K-1}\subset \mathcal{H}$ with $\sum_{k=0}^{K-1}\Vert \phi_k\Vert_{\mathcal{H}}^p=1$, we call a piecewise defined function $a:[0,T)\rightarrow \mathcal{H}$,
 \begin{align*}
     a=\sum_{k=1}^K\mathds{1}_{[t_{k-1},t_k)}\phi_{k-1}
 \end{align*}
 a $U^p$-atom. Therefore, the atomic space $U^p([0,T);\mathcal{H})$ is defined to be the set of all functions $u:[0,T)\rightarrow \mathcal{H}$ of the form
 \begin{align*}
  u=\sum_{j=1}^{\infty}\lambda_ja_j\quad \text{for}\,\,U^p-\text{atoms} \,\,a_j,\,\,\{\lambda_j\}\in \ell^1(\mathbb{C}),
 \end{align*}
with norm
 \begin{align*}
     \Vert u\Vert_{U^p}:=\inf\Big\{\sum_{j=1}^{\infty}|\lambda_j|: u=\sum_{j=1}^{\infty}\lambda_ja_j\,\,\text{with}\,\,\{\lambda_j\}\in \ell^1(\mathbb{C})\,\,\text{and}\,\,U^p\,\,\text{atom}\,\,a_j\Big\}. \end{align*}
 \end{definition}
 \begin{definition}
 Let $1\leq p<\infty$. \begin{itemize}
 \item[(i)] We define $V^p([0,T);\mathcal{H})$ as the normed space of all functions $v: [0,T)\rightarrow \mathcal{H}$ with 
 \begin{align}\label{V^pnorm}
     \Vert v\Vert_{V^p}:=\sup_{\{t_k\}_{k=0}^K\in\mathcal{Z}}\Big(\sum_{k=1}^K\Vert v(t_k)-v(t_{k-1})\Vert_{\mathcal{H}}^p\Big)^{\frac{1}{p}}<\infty.
 \end{align}
  \item[(ii)] The space $V^p_{rc}([0,T);\mathcal{H})$ denotes the closed subspace of all right-continuous $v\in V^p$ such that $v(0)=0$, endowed with the same norm \eqref{V^pnorm}.
 \end{itemize}
 \end{definition}
 \begin{remark}\label{embeddings}
 For $1\leq p < q<\infty$, the following embeddings hold:
 \begin{align}\label{embeddingUV}
  U^p([0,T);\mathcal{H})\hookrightarrow V^p_{rc}([0,T);\mathcal{H})\hookrightarrow U^q([0,T);\mathcal{H})\hookrightarrow L^{\infty}([0,T);\mathcal{H}).   
 \end{align}
 Thus, the functions in $U^p([0,T);\mathcal{H})$ are considered as right continuous, and $u(0)=0$ for all $u\in U^p([0,T);\mathcal{H})$. 
 \end{remark}
 
 In what follows, we shall consider $\mathcal H = H^s =  H^s(\T^d)$ for some $s \in \R$.
 
 \begin{definition}
For $s\in\mathbb{R}$, we define $U_{\Delta_{\pm}}^pH^s$ and $V_{\Delta_{\pm}}^pH^s$ as spaces of all functions $u: [0,T)\rightarrow H^s(\mathbb{T}^d )$ such that the map $t \rightarrow e^{-it\Delta_{\pm}} u(t)$ is in $U^p([0,T);H^s)$ and $V^p([0,T);H^s)$ respectively, with norms 
\begin{align*}
    \Vert u\Vert_{U^p_{\Delta_{\pm}}H^s}:= \Vert e^{-it\Delta_{\pm}}u\Vert_{U^p([0,T);H^s)},\quad \Vert u\Vert_{V^p_{\Delta_{\pm}}H^s}:= \Vert e^{-it\Delta_{\pm}}u\Vert_{V^p([0,T);H^s)}.
\end{align*}
 \end{definition}
 
 \begin{definition}
 Let $s\in\mathbb{R}$. We define $X^s$ and $Y^s$ as spaces of all functions $u: [0,T)\rightarrow H^s(\mathbb{T}^d )$ such that for every $k \in \mathbb{Z}^d$, the map $t\rightarrow \widehat{e^{-it\Delta_{\pm}}u(t)}(k)$  is in $U^2([0,T))$ and $V^2_{\text{rc}}([0,T))$ respectively, for which the norms
 \begin{equation*}
 \begin{aligned}
 \Vert u\Vert_{X^s([0,T))}&:=\Big(\sum_{k \in\mathbb{Z}^d}\langle k \rangle^{2s}\Vert\widehat{e^{-it\Delta_{\pm}}u(t)}(k)\Vert_{U^2_t}^2\Big)^{\frac{1}{2}}, \\ 
 \Vert u\Vert_{Y^s([0,T))}&:=\Big(\sum_{k\in\mathbb{Z}^d}\langle k\rangle^{2s}\Vert\widehat{e^{-it\Delta_{\pm}}u(t)}(k)\Vert_{V^2_t}^2\Big)^{\frac{1}{2}},
 \end{aligned}
 \end{equation*}
 are finite.
 \end{definition} 
 
 We have the following continuous embeddings
 \begin{align}\label{U^2trXsembed}
  U^2_{\Delta_{\pm}}H^s\hookrightarrow X^s \hookrightarrow Y^s \hookrightarrow V^2_{\Delta_{\pm}}H^s.  
 \end{align} 
\begin{lemma}\label{linearestlemma}
 Let $s\geq 0$ and $f\in H^s(\mathbb{T}^d )$. Then,
 \begin{equation*}
 \begin{aligned}
  \Vert e^{it\Delta_{\pm}}f\Vert_{Y^s([0,T))}&\lesssim \Vert e^{it\Delta_{\pm}}f\Vert_{X^s([0,T))} \lesssim \Vert f\Vert_{H^s}, \\
 \Vert u \Vert_{L^{\infty}_tH^s_x([0,T)\times \mathbb{T}^d )}&\lesssim  \Vert u\Vert_{Y^s([0,T))}\lesssim  \Vert u\Vert_{X^s([0,T))}.
 \end{aligned}
 \end{equation*}
 \end{lemma}
 \begin{lemma}\label{lemmaduality}
 Let $s\geq 0$ and $f\in L^1([0,T);H^s(\mathbb{T}^d ))$. Then,
 \begin{equation*} 
 \begin{aligned}
  \Big\Vert \int_0^te^{i(t-s)\Delta_{\pm}}f(s)ds\Big\Vert_{X^s([0,T))}\lesssim \sup_{\Vert v\Vert_{Y^{-s}([0,T))}\leq 1}\Big|\int_0^T\int_{\mathbb{T}^d }f(t,x)\overline{v(t,x)}dxdt\Big|.
 \end{aligned}
 \end{equation*}
 \end{lemma}

\begin{lemma}\label{LEM:Stri_Y} 
For all $N\geq 1$ and admissible pairs $(d,p)$ given in Definition \ref{DEF:adm}, we have
\begin{equation}\label{StricartzwithN}
\begin{aligned}
  \Vert P_{\leq N}u\Vert_{L^p([0,T)\times \mathbb{T}^d )} \lesssim N^{\frac{d}{2}-\frac{d+2}{p}}\Vert P_{\leq N}u\Vert_{U^p_{\Delta_{\pm}}L^2}\lesssim N^{\frac{d}{2}-\frac{d+2}{p}}\Vert P_{\leq N}u\Vert_{Y^0([0,T))}. 
\end{aligned}
\end{equation}
\end{lemma}

\begin{proof}
We first note that $p > 2$ from \eqref{admissible}.
The first estimate in \eqref{StricartzwithN} (resp., \eqref{StricartzwithCcubes}) follows from \eqref{BDstrcest2} (resp., \eqref{StricartzwithC}) by using the atomic structure of $U^p$, see \cite[Proposition 2.19]{HHK09} and \cite[Corollary
3.2]{HTT11} for details. Recall that the induced norm on $V^p_{\text{rc}}([0,T);L^2)$ is the same as that for $V^p([0,T);L^2)$. Then, the second estimates are easy consequences of the embeddings \eqref{embeddingUV} and \eqref{U^2trXsembed}.
\end{proof}

\section{Scale invariant Strichartz estimates}\label{scinvstr}
In this section, our goal is to remove the $N^{\eps}$ loss from the Strichartz estimate \eqref{BDstrcest1} of Theorem \ref{THM:Stri_T}, when $p> \frac{2(d+2-\delta (j_0))}{d-\delta (j_0)}$, thus proving Theorem \ref{THM:scale2}. 
We shall follow the same strategy as in \cite{KV14}.
Then, we use the hyperbolic Galilean transformation, determined by the hyperbolic Laplacian operator $\Delta_\pm$, to relocate the centre of the frequency localisation.

\subsection{The \texorpdfstring{$\eps$}{Lg}-removal argument}
\label{epsilonremoval}

In what follows, we focus on proving Theorem~\ref{THM:scale2}. Let us define the convolution kernels associated with the hyperbolic propagator as
\begin{align*}
K_N^\pm(t,x) &:= \big[e^{it\Delta_{\pm}} P_{\leq N}\delta_0\big](x) \\
&= \sum_{k \in \mathbb{Z}^d} \exp\big(2\pi i x \cdot k + 2\pi i t H_\varepsilon(k)\big) \cdot \prod_{j=1}^d \varphi\bigg(\frac{k_j}{N}\bigg),
\end{align*}
where $P_{\leq N}$ denotes the frequency projection to modes $|k| \leq N$, $H_{\varepsilon}(k)$ is as in \eqref{Heps}. 
In the special case when $\delta (j_0) = 0$, we simplify notation by writing $K_N^\pm$ as simply $K_N$. 
For this kernel, we establish the following dispersive estimate.

\begin{lemma}\label{LEM:disK}
Let $0\leq a_j\leq q_j<N$ such that $(a_j,q_j)=1$ and $|\eps_j t-\frac{a_j}{q_j}|\leq \frac{1}{q_jN}$. Then, we have
\begin{align}\label{K_Nestimate}
    |K_N^\pm (t,x)|\lesssim \prod_{j=1}^d\frac{N}{\sqrt{q_j}(1+N|\eps_j t-\frac{a_j}{q_j}|^{\frac{1}{2}})}
\end{align}
uniformly for $t\in[0,1]$.
\end{lemma}

For the elliptic case, i.e., when $\delta (j_0) = 0$, Lemma~\ref{LEM:disK} is proved in \cite[Lemma~3.18]{B93} using Weyl's method. 
See also \cite[Lemma 2.2]{KV14}.
The hyperbolic case of Lemma~\ref{LEM:disK} follows from the elliptic case via a symmetry argument, which we present below.

\begin{proof}
We note that
\begin{align}\label{Keq}
\begin{split}
|K_N^\pm (t,x)|&=\Bigg| \prod_{j=1}^{j_0} \sum_{k_j\in\mathbb{Z}}\varphi \bigg(\frac{k_j}N\bigg) e^{2\pi i[x_j k_j-t\eps_j k_j^2]} \prod_{j=j_0+1}^d \sum_{k_j \in\mathbb{Z}} \varphi \bigg(\frac{k_j}N\bigg) e^{2\pi i[x_jk_j+t\eps_j k_j^2]}\Bigg|\\
&=\Bigg| \prod_{j=1}^{j_0} \sum_{k_j\in\mathbb{Z}}\varphi \bigg(\frac{k_j}N\bigg) e^{- 2\pi i[x_j k_j + t\eps_j k_j^2]} \Bigg | \Bigg| \prod_{j=j_0+1}^d \sum_{k_j \in\mathbb{Z}} \varphi \bigg(\frac{k_j}N\bigg) e^{2\pi i[x_jk_j+t\eps_j k_j^2]}\Bigg|\\
&=\Bigg| \prod_{j=1}^{j_0} \sum_{k_j\in\mathbb{Z}}\varphi \bigg(\frac{k_j}N\bigg) e^{2\pi i[x_j k_j+ t\eps_j k_j^2]} \Bigg | \Bigg| \prod_{j=j_0+1}^d \sum_{k_j \in\mathbb{Z}} \varphi \bigg(\frac{k_j}N\bigg) e^{2\pi i[x_jk_j+t\eps_j k_j^2]}\Bigg|\\
&=\Bigg| \prod_{j=1}^{d} \sum_{k_j\in\mathbb{Z}}\varphi \bigg(\frac{k_j}N\bigg) e^{2\pi i[x_j k_j+ t\eps_j k_j^2]} \Bigg | = |K_N (t,x)|,
\end{split}
\end{align}
where we use the fact that $\varphi$ is radial in the third step.
Then, \eqref{K_Nestimate} follows from \cite[Lemma 3.18]{B93}.
We thus finish the proof.
\end{proof}

Following notations in \cite{KV14}, define the set
\begin{align}
\label{T}
\mathcal{T}:=\{t\in[0,1]: q_jN^2\big|\eps_j t-a_j/q_j\big|\leq N^{2\sigma}\,\,\text{for some}\,j,q_j\leq N^{2\sigma},\,\text{and}\,(a_j,q_j)=1\}   
\end{align}
on which $K_N(t,x)$ is large. So we define the large part of the convolution kernel by
\begin{align*}
\widetilde{K}_N^\pm (t,x):=\mathds{1}_{\mathcal{T}}(t)K_N^\pm (t,x).
\end{align*}
The following result follows from \eqref{Keq} and  \cite[Proposition 2.3]{KV14}.
The estimate \eqref{K_Nestimate} together with \eqref{T} gives that
\begin{align}\label{diffT}
    |K_N^\pm (t,x)-\widetilde{K}_N^\pm (t,x)|=|\mathds{1}_{[0,1]\setminus\mathcal{T}}(t)K_N^\pm (t,x)|\lesssim N^{n(1-\sigma)}.
\end{align}
Furthermore, by choosing $\s$ in \eqref{T} small enough, we have the following convolution estimate for the kernel $\wt K_N^\pm $.

\begin{proposition}\label{mainprop}
Assume $r>\frac{2(d+2)}{d}$. Then
\begin{align}\label{tildeKernel}
\Vert \widetilde{K}_N^\pm *F\Vert_{L^r_{t,x}([0,1]\times \mathbb{T}^d)}\lesssim N^{d-\frac{2(d+2)}{r}}\Vert F\Vert_{L^{r'}_{t,x}([0,1]\times \mathbb{T}^d)},
\end{align}
provided $\sigma $ in \eqref{T} is small enough (depending only on $n, p$).
\end{proposition}
\begin{proof}
The estimate \eqref{tildeKernel} follows from Lemma \ref{LEM:disK}, and the proof for \cite[Proposition 2.3]{KV14} which is divided into two crucial spatial (\cite[Lemma 2.4]{KV14}) and temporal convolution (\cite[Lemma 2.5]{KV14}) estimates. The only minor difference arises in the spatial convolution estimate where we make use of Lemma \ref{LEM:disK} instead, yet still, gives the same contribution.  
\end{proof}

With these preparations, we are ready to present the proof of Theorem \ref{THM:scale2}.
We remark that the $\varepsilon$-removal argument for periodic Strichartz estimates was introduced in \cite{B93}; see also \cite{OhWang,KV14}.
For completeness, we present the proof, following the argument in \cite{KV14} for the elliptic case.
The main difference in the hyperbolic case is that we use Theorem \ref{THM:Stri_T} with $\delta (j_0) \neq 0$, rather than $\delta (j_0) = 0$, as well as the observations in Lemma \ref{LEM:disK} and Proposition \ref{mainprop}.

\begin{proof}[Proof of Theorem \ref{THM:scale2}]
Normalize $\Vert f\Vert_{L^2(\mathbb{T}^d )}=1$ so that via Bernstein's inequality we get
\begin{align*}
    \Vert e^{it\Delta_{\pm}}P_{\leq N}f\Vert_{L^{\infty}_{t,x}([0,1]\times \mathbb{T}^d )}\leq CN^{\frac{d}{2}}
\end{align*}
for some $C>0$. Fix $p>p^*:=\frac{2(d+2-\delta (j_0))}{d-\delta (j_0)}$, then the inequality above allows us to write
\begin{align*}
 \Vert e^{it\Delta_{\pm}}P_{\leq N}f\Vert_{L^{p}_{t,x}([0,1]\times \mathbb{T}^d )}^p=\int_0^{CN^{\frac{d}{2}}}p\mu^{p-1}\big|\{(t,x)\in [0,1]\times\mathbb{T}^d :|(e^{it\Delta_{\pm}}P_{\leq N}f)(x)|>\mu\}\big|d\mu.\end{align*}
 For small $\delta>0$ to be specified later, we use Chebyshev's inequality and Theorem \ref{THM:Stri_T} with $p^*$ to get
 \begin{equation}\label{firstest}
\begin{aligned}
 \int_0^{N^{\frac{d}{2}-\delta}}&p\mu^{p-1}\big|\{(t,x)\in [0,1]\times\mathbb{T}^d :|(e^{it\Delta_{\pm}}P_{\leq N}f)(x)|>\mu\}\big|d\mu \\&\lesssim N^{p^*(\frac{d}{2}-\frac{n+2}{p^*}+\eps)}\int_0^{N^{\frac{d}{2}-\delta}}\mu^{p-p^*-1}d\mu\\&\lesssim N^{p(\frac{d}{2}-\frac{d+2}{p})+\eps p^*-\delta(p-p^*)}\lesssim N^{p(\frac{d}{2}-\frac{d+2}{p})}  
\end{aligned}
\end{equation}
as long as $\eps<\frac{\delta(p-p^*)}{p^*}$. For fixed $\mu>N^{\frac{d}{2}-\delta}$, we shall denote 
\begin{align*}
    \Omega=\{(t,x)\in [0,1]\times\mathbb{T}^d :|(e^{it\Delta_{\pm}}P_{\leq N}f)(x)|>\mu\}.
\end{align*} Hence, it is left to show that
\begin{align*}
I:= \int_{N^{\frac{d}{2}-\delta}}^{CN^{\frac{d}{2}}}p\mu^{p-1}\big|\Omega\big|d\mu \lesssim N^{p(\frac{d}{2}-\frac{d+2}{p})}, 
\end{align*}
which together with \eqref{firstest} gives the desired estimate \eqref{BDstrcest2}. For this purpose, we introduce the set \begin{align*}
    \Omega_{\omega}=\{(t,x)\in [0,1]\times\mathbb{T}^d :\text{Re}(e^{i\omega}e^{it\Delta_{\pm}}P_{\leq N}f)(x)>\frac{\mu}{2}\}
\end{align*}
that satisfies $|\Omega|\leq 4|\Omega_{\omega}|$ for some $\omega\in \{0, \frac{\pi}{2}, \pi, \frac{3\pi}{2}\}$. Therefore, it suffices to estimate $\Omega_{\omega}$. By Cauchy-Schwarz inequality, we have
\begin{equation}\label{Omega_omegaest}
\begin{aligned}
\mu^2|\Omega_{\omega}|^2&\lesssim \langle e^{it\Delta_{\pm}}P_{\leq N}f, \mathds{1}_{\Omega_{\omega}} \rangle_{L^2_{t,x}}^2\\&=\langle f, e^{-it\Delta_{\pm}}P_{\leq N}\mathds{1}_{\Omega_{\omega}} \rangle_{L^2_{t,x}}^2\\&\lesssim \langle \mathds{1}_{\Omega_{\omega}}, K_N^\pm*\mathds{1}_{\Omega_{\omega}}\rangle_{L^2_{t,x}}.
\end{aligned}
\end{equation}
We write $K_N^\pm=[K_N^\pm-\widetilde{K}_N^\pm]+\widetilde{K}_N^\pm$ to estimate the resulting bound in \eqref{Omega_omegaest}. To estimate the first term, we require the following lemma.

Applying Hölder's inequality, Young's inequality, and \eqref{diffT} we obtain the bound
\begin{equation}\label{K_N-K_Ntilde}
\begin{aligned}
 |\langle \mathds{1}_{\Omega_{\omega}},(K_N^\pm-\widetilde{K}_N^\pm)*\mathds{1}_{\Omega_{\omega}}\rangle_{L^2_{t,x}}|&\lesssim \Vert \mathds{1}_{\Omega_{\omega}}\Vert_{L^1_{t,x}}\Vert (K_N^\pm-\widetilde{K}_N^\pm)*\mathds{1}_{\Omega_{\omega}}\Vert_{L^{\infty}_{t,x}} \lesssim |\Omega_{\omega}|^2N^{d(1-\sigma)}. 
\end{aligned}
\end{equation} 
Now fix $r\in (p^*,p)$ and note that $p^*>\frac{2(d+2)}{d}$. Then, by choosing $\sigma$ sufficiently small, we apply Proposition \ref{mainprop} to obtain
\begin{equation}\label{K_Ntilde}
\begin{aligned}
 |\langle \mathds{1}_{\Omega_{\omega}},\widetilde{K}_N^\pm*\mathds{1}_{\Omega_{\omega}}\rangle_{L^2_{t,x}}|&\lesssim \Vert \mathds{1}_{\Omega_{\omega}}\Vert_{L^{r'}_{t,x}}\Vert \widetilde{K}_N^\pm*\mathds{1}_{\Omega_{\omega}}\Vert_{L^{r}_{t,x}} \lesssim |\Omega_{\omega}|^{\frac{2}{r'}}N^{d-\frac{2(d+2)}{r}}. 
\end{aligned}
\end{equation}
Combining \eqref{Omega_omegaest}, \eqref{K_N-K_Ntilde}, \eqref{K_Ntilde}, and choosing $\delta\ll \frac{d\sigma}{2}$, we see that
\begin{align*}
  \mu^2|\Omega_{\omega}|^2&\lesssim |\Omega_{\omega}|^{\frac{2}{r'}}N^{d-\frac{2(d+2)}{r}}
\end{align*}
since the contribution coming from \eqref{K_N-K_Ntilde} is much smaller than $\mu^2|\Omega_{\omega}|^2$ under the assumption $\delta\ll \frac{d\sigma}{2}$ and $\mu>N^{\frac{d}{2}-\delta}$. Therefore, for $r\in (p^*,p)$, this implies that
\begin{equation*}
\begin{aligned}
I&\leq 4\int_{N^{\frac{d}{2}-\delta}}^{CN^{\frac{d}{2}}}p\mu^{p-1}\big|\Omega_{\omega}\big|d\mu&\lesssim N^{\frac{r}{2}(d-\frac{2(d+2)}{r})}\int_{N^{\frac{d}{2}-\delta}}^{CN^{\frac{d}{2}}}\mu^{p-r-1}d\mu\lesssim N^{p(\frac{d}{2}-\frac{d+2}{p})},
\end{aligned}
\end{equation*}
which together with \ref{firstest} finishes the proof.
\end{proof}

\subsection{Hyperbolic Galilean boost}

We denote by $\mathscr{C}_N$ the collection of cubes $C\subset \mathbb{Z}^d$ of side-length $N\geq 1$ with arbitrary center and orientation.
By a hyperbolic type Galilean transformation, Theorem \ref{THM:scale2} leads to the following consequence.

\begin{lemma}\label{LEM:Stri_Ga} 
Let $0< T \leq 1$. For all $C\in \mathscr{C}_N$ and admissible pair $(d,p)$, we have
\begin{equation}\label{StricartzwithC}
\begin{aligned}
  \Vert P_Ce^{it\Delta_{\pm}}\phi\Vert_{L^p([0,T]\times \mathbb{T}^d )} \lesssim N^{\frac{d}{2}-\frac{d+2}{p}}\Vert P_C\phi\Vert_{L^2(\mathbb{T}^d )}. 
\end{aligned}
\end{equation}
\end{lemma}
\begin{proof}
Let $C\in \mathscr{C}_N$ be the cube with center $r = (r_{1},\cdots, r_d) \in \R^d$, and 
\begin{align}\label{Heps}
 H_{\eps}(k) = - \sum_{j=1}^{j_0} \eps_j k_j^2 + \sum_{j=j_0+1}^d \eps_j k_j^2,
\end{align}
where $j_0 \in \{1, \cdots, d-1\}$ and $\eps = (\eps_1,\cdots,\eps_d) \in \R^d_+$.
Let us denote $C':=C-r\in \mathscr{C}_N$ by the cube with center at the origin.
 To shift the center of the frequency localization, we make use of the following hyperbolic type Galilean transform (see also \cite[Section 3]{YW13_2})
\begin{align*}
    x\cdot k + t H_{\eps}(k) = x\cdot r + t H_{\eps} (r)+(x+2t\overline{\eps r}) \cdot (k-r)+tH_{\eps}(k-r)
    \end{align*}
where 
$$
\overline{\eps  r}=(-\eps _1r_{1},...,-\eps_{j_0} r_{j_0}, \eps _{j_0+1}r_{j_0+1},...,\eps_{d}r_{ d}).
$$ 
Therefore, by setting $\phi_0(x):=e^{-ix\cdot r}\phi(x)$, we obtain  $P_{C}\phi\,(x)=e^{ix\cdot r}P_{C_0}\phi_0\,(x)$ and
\begin{equation*}\label{e^itexpress}
\begin{aligned}
   e^{it\Delta_{\pm}} P_{C}\phi\,(t,x) & =  \sum_{k \in C} e^{2\pi i[x\cdot k+t H_{\eps }(k)]} \widehat{\phi}(k)\\
   &=e^{2\pi i(x\cdot r+tH_{\eps }(r))}\sum_{k \in C}e^{2\pi i[(x+2t\overline{\eps  r})\cdot (k-r)+tH_{\eps }(k-r)]} \widehat{\phi}(k)\\&=e^{2\pi i(x\cdot r+tH_{\eps }(r))}\sum_{k \in C_0}e^{2\pi i[(x+2t\overline{\eps  r})\cdot k)+tH_{\eps }(k)]} \widehat{\phi_0}(k)\\&=e^{2\pi i(x\cdot r+tH_{\eps }(r))}e^{it\Delta_{\pm}}P_{C_0}\phi_0(t, x+2t\overline{\eps  r}).
    \end{aligned}      
    \end{equation*}
As a result, the estimate \eqref{StricartzwithC} follows from Theorem \ref{THM:scale2}.
\end{proof}

In view of Lemma \ref{LEM:Stri_Y},
we have the following corollary of Lemma \ref{LEM:Stri_Ga}.

\begin{corollary}\label{COR:Stri_GY}
For all $C\in \mathscr{C}_N$ and admissible pair $(d,p)$, we have
\begin{equation}
\label{StricartzwithCcubes}
\begin{aligned}
  \Vert P_C u\Vert_{L^p([0,T)\times \mathbb{T}^d )} \lesssim N^{\frac{d}{2}-\frac{d+2}{p}}\Vert P_C u\Vert_{U^p_{\Delta_{\pm}}L^2}\lesssim N^{\frac{d}{2}-\frac{d+2}{p}}\Vert P_C u\Vert_{Y^0([0,T))}. 
\end{aligned}
\end{equation}
\end{corollary}

The proof of Corollary \ref{COR:Stri_GY} follows the same lines as the proof of Lemma \ref{LEM:Stri_Y}, replacing Theorem \ref{THM:scale2} with Lemma \ref{LEM:Stri_Ga}. We omit the details.

\section{Nonlinear estimates}\label{SectNonlinear}
In this section, based on the critical function spaces theory discussed in Section \ref{preliminaries}, we prove nonlinear estimates associated with \eqref{eq:NLS}, which will then be utilized in order to establish Theorem \ref{THM:main}. 
\subsection{Estimate for subcritical spaces}
The following estimate is a crucial ingredient for the well-posedness of the three-dimensional cubic HNLS equation \eqref{eq:NLS} in the subcritical regime.
In the following, we present the proof of Theorem \ref{THM:main0} only for the case $d = 3$,
since the arguments for the other cases are similar.
\begin{lemma}\label{PROP:T2m}
 Fix $s> s_c (3,1)=\frac{1}{2}$. Let $j_0\in \{1,2\}$ and $0<T\leq 1$. Then we have
\begin{align}\label{subcrest}
\Big|\int_0^T\int_{\mathbb{T}^3}\widetilde{u}_1\widetilde{u}_2\widetilde{u}_3\overline{v}\,dxdt\Big|\lesssim \Vert v\Vert_{Y^{-s}}\Vert u_1\Vert_{Y^{s}}\Vert u_2\Vert_{Y^{s}}\Vert u_3\Vert_{Y^{s}}
\end{align}
where $\widetilde{u}_j\in \{u_j, \overline{u}_j\}$.
\end{lemma}
\medskip
To prove Lemma \ref{PROP:T2m}, we need the following bilinear Strichartz estimate.
\begin{lemma}
    Let $T \le 1$. Then for any $\eps>0$ and every $1 \le N_2 \le N_1$ we have
    \begin{align}
        \label{bilinear}
        \| P_{N_1} u_1 P_{N_2} u_2\|_{L^2_{t,x} ([0,T)\times \T^3)} \les N_2^{\frac{1}{2}+\eps} \| u_1\|_{Y^0 ([0,T))} \| u_2\|_{Y^0 ([0,T))}.
    \end{align}
   \end{lemma}
The proof of \eqref{bilinear} is similar to \cite[Lemma 3.1]{KV14}, therefore we omit. However, the factor $N_2^{\frac{1}{2}+\eps}$ in \eqref{bilinear} comes from the non-scale-invariant Strichartz estimates of Theorem \ref{THM:Stri_T}. Indeed, in view of \eqref{d(j_0)}, we see that $d(1)=d(2)=1$, which then allows us to take advantage of the $L^4_{t,x}$ Strichartz inequality resulting from Theorem \ref{THM:Stri_T}:
\begin{align*}
\Vert e^{it\Delta_{\pm}}P_{C}\phi\Vert_{L^4_{t,x}([0,1]\times\mathbb{T}^3)}\lesssim N^{\frac{1}{4}+\delta}\Vert P_{C}\phi\Vert_{L^2(\mathbb{T}^3 )},
\end{align*}
for each $\delta>0$ and all $C\in \mathscr{C}_N$. Therefore, transferring the above estimate to an estimate basically of the form \eqref{StricartzwithCcubes} (in this case with loss $N^{\frac{1}{4}+\delta}$) and using the cube decomposition as usual give rise to the bilinear estimate \eqref{bilinear}.
\begin{proof}[Proof of Lemma \ref{PROP:T2m}]
By symmetry, it suffices to assume $N_1 \ge N_2 \ge N_3$. Then it follows that
\begin{align}
\begin{split}
I:=\Big|\int_0^T\int_{\mathbb{T}^3}\widetilde{u}_1\widetilde{u}_2\widetilde{u}_3\overline{v}\,dxdt\Big| & \les \sum_{\substack{N_0,N_1,N_2,N_3\\
N_1 \ge N_2 \ge N_3}} \|P_{N_0}vP_{N_1}u_{1} P_{N_2}u_{2} P_{N_3}u_{3} \|_{L^1_{t,x}}.
\end{split}
\label{termI}
\end{align}
Given $\eps>0$, set $s=\frac{1}{2}+2\eps$. Let us first consider the case when $N_0\gg N_2$. Then we have $N_0 \sim N_1$. From \eqref{termI} and \eqref{bilinear}, we obtain 
\begin{align}
\begin{split}
I & \les \sum_{\substack{N_0,N_1,N_2,N_3\\
N_0 \sim N_1 \gg N_2 \ge N_3}} \|P_{N_1}u_{1} P_{N_2}u_{2}\|_{L^2_{t,x}  ([0,T)\times \T^3)} \| P_{N_3}u_{3} P_{N_0}v\|_{L^2_{t,x}  ([0,T)\times \T^3)} \\
& \les \sum_{\substack{N_0,N_1,N_2,N_3\\
N_0 \sim N_1 \gg N_2 \ge N_3}} (N_2 N_3)^{\frac{1}{2}+\eps}  \| P_{N_0}v\|_{Y^0}\prod_{j=1}^3 \| P_{N_j}u_{j}\|_{Y^0 } \\
& \les \sum_{\substack{N_0,N_1,N_2,N_3\\
N_0 \sim N_1 \gg N_2 \ge N_3}} (N_2 N_3)^{-\eps}  \| P_{N_0}v\|_{Y^{-s} }  \prod_{j=1}^3 \| P_{N_j}u_{j}\|_{Y^{s} }  \\
& \les \Big( \sum_{N_0 \sim N_1} \|  P_{N_0}v\|_{Y^{-s} }  \|  P_{N_1}u_{1}\|_{Y^{s} } \Big) \prod_{j=2}^3 \Big( \sum_{N_j} N_j^{-\eps} \| P_{N_j}u_{j}\|_{Y^{s}} \Big)  \\
& \les \|v\|_{Y^{-s}}\bigg( \prod_{j = 1}^3 \|P_{N_j}u_j\|_{Y^{s}} \bigg).  
\end{split}
\label{termI1}
\end{align}
Next, let us assume $N_0 \les N_2$, which implies $N_1 \sim N_2$. Thus similarly as above we get
\begin{align}
\begin{split}
I & \les \sum_{\substack{N_0,N_1,N_2,N_3\\
N_1 \sim N_2 \ges N_0, N_3}} \|P_{N_1}u_{1} P_{N_3}u_{3}\|_{L^2_{t,x}} \|P_{N_0} v P_{N_2}u_{2}\|_{L^2_{t,x}} \\
& \les \sum_{\substack{N_0,N_1,N_2,N_3\\
N_1 \sim N_2 \ges N_0, N_3}} (N_0 N_3)^{\frac{1}{2}+\eps}  \| P_{N_0}v\|_{Y^0}\prod_{j=1}^3 \| P_{N_j}u_{j}\|_{Y^0} \\
& \les \sum_{\substack{N_0,N_1,N_2,N_3\\
N_1 \sim N_2 \ges N_0, N_3}} (N_1 N_2)^{-s} N_3^{-\eps} N_0^{2s-\eps}   \Big( \prod_{j=1}^3 \| P_{N_j}u_{j}\|_{Y^{s} } \Big) \|P_{N_0}v\|_{Y^{-s} }\\
& \les \sum_{\substack{N_0,N_1,N_2,N_3\\
N_1 \sim N_2 \ges N_0, N_3}}  (N_0N_3)^{-\eps}  \Big( \prod_{j=1}^3 \| P_{N_j}u_{j}\|_{Y^{s} } \Big) \| P_{N_0}v\|_{Y^{-s}} \\
& \les \Big( \sum_{N_1 \sim N_2} \| P_{N_1}u_{1}\|_{Y^{s}}  \| P_{N_2}u_{2}\|_{Y^{-s} } \Big)\Big(\sum_{N_3} N_3^{-\eps} \| P_{N_3}u_{3}\|_{Y^{s}} \Big) \Big(\sum_{N_0} N_0^{-\eps} \| P_{N_0}v\|_{Y^{-s}} \Big)   \\
& \les \|v\|_{Y^{-s}}\bigg( \prod_{j = 1}^3 \|u_j\|_{Y^{s}} \bigg).
\end{split}
\label{termI2}
\end{align}

\noi 
Thus by combining \eqref{termI1} and \eqref{termI2} the estimate \eqref{subcrest} follows.
\end{proof}

\subsection{Estimates for the critical spaces}
We start with the following estimate that covers numerous cases.
\begin{lemma}\label{lemma1}
 Assume $d=2$, $m\geq 4$; $d\geq 3$, $m\geq 2$, and $s\geq s_c (d,m)$, $0<T\leq 1$. Let $j_0\in \{1,...,d-1\}$, then we have
\begin{align*}
\Big|\int_0^T\int_{\mathbb{T}^d }\prod_{j=1}^{2m+1}\widetilde{u}_j\overline{v}\,dxdt\Big|\lesssim \Vert v\Vert_{Y^{-s}}\Vert u_1\Vert_{Y^{s}}\prod_{j=2}^{2m+1}\Vert u_j\Vert_{Y^{s_c (d,m)}}
\end{align*}
where $\widetilde{u}_j\in \{u_j, \overline{u}_j\}$.
\end{lemma}
\begin{proof}
 Let $p$ satisfy the following  
\begin{equation}\label{psatisfy}
\begin{aligned}
6<p<\frac{16m}{2m+2}\quad \text{for}\,\,d=2,\,\,m\geq 4,\\
 2+\frac{8}{d}<p<\frac{4m(d+2)}{dm+2} \quad \text{for}\,\,d\geq 4\,\, \text{even} ,\,\,m\geq 2,\\ 
  2+\frac{8}{d+1}<p<\frac{4m(d+2)}{dm+2} \quad \text{for}\,\,d\geq 3\,\, \text{odd},\,\,m\geq 2,
\end{aligned}
\end{equation}
where the lower bounds are due to \eqref{worstcase} so that each $(d,p)$ is admissible and the upper bounds we choose ensure that
\begin{align}\label{s_nm<0}
  d-\frac{2(d+2)}{p}-s_c (d,m)<0  
\end{align}
which will be used to give the desired result. We write $p_{d,m}:=m(d+2)$ to denote the exponent  associated to critical scaling exponent, that is, $s_c (d,m)=\frac{d}{2}-\frac{d+2}{p_{d,m}}$, so that via Lemma \ref{LEM:Stri_Y} one has
\begin{equation}\label{pnmsnm}
\begin{aligned}
  \Vert P_Nu\Vert_{L^{p_{d,m}}([0,T]\times \mathbb{T}^d )} \lesssim N^{s_c (d,m)}\Vert P_Nu\Vert_{Y^0}. 
\end{aligned}
\end{equation}
In what follows, we employ H\"older's inequality with exponents $p$ satisfying \eqref{psatisfy}, $p_{d,m}$ as introduced above, and $q$ for $(d,q)$ being admissible such that
\begin{align}\label{pp_nmq2}
    \frac{2}{p}+\frac{m-2}{p_{d,m}}+\frac{1}{q}=\frac{1}{2}.
\end{align}
The admissibility of $(d,q)$ is verified by using \eqref{pp_nmq2} and the upper bounds in \eqref{psatisfy}, indeed, 
\begin{equation*}
\begin{aligned}
q>p_{d,m}=md+2m> \frac{2(d+2-\delta (j_0))}{d-\delta (j_0)}   
\end{aligned}
\end{equation*}
for any $j_0\in\{1,...,d-1\}$. Also, in view of \eqref{pp_nmq2} and \eqref{s_nm<0}, we observe that 
\begin{align}\label{necessarydelta}
  \frac{d}{2}-\frac{d+2}{q}-s_c (d,m)=-d+\frac{2(d+2)}{p}+s_c (d,m)>0.
\end{align}
Let $\delta:=-d+\frac{2(d+2)}{p}+s_c (d,m)$. By symmetry, it suffices to consider the following two cases.
\vspace{0.12cm}
\noindent 
\\{\bf Case A.} $N_0\sim N_1\geq N_2\geq \cdots\geq N_{2m+1}$.  \vspace{0.12cm}
\\
Consider partitions of $\mathbb{Z}^d$ into cubes $C_{k_2}\in\mathscr{C}_{N_2}$ and $C_{k_3}\in\mathscr{C}_{N_3}$ respectively, then 
$$P_{C_{k_2}}P_{N_0}v\prod_{j=1}^{m}P_{N_{2j}}u_{2j}\quad \text{and}\,\,\, P_{C_{k_3}}P_{N_{1}}u_{1}\prod_{j=1}^{m}P_{N_{2j+1}}u_{2j+1}$$ are almost orthogonal in $L^2_x$
because there are a bounded number of cubes of side length $\sim N_2$ and of length $\sim N_3$ which cover the spatial Fourier support of $P_{N_2}u_2\prod_{j=2}^{m}P_{N_{2j}}u_{2j}$ and $P_{N_3}u_3\prod_{j=2}^{m}P_{N_{2j+1}}u_{2j+1}$ respectively. Then by almost orthogonality, H\"older's inequality with exponents in \eqref{pp_nmq2}, the estimates \eqref{StricartzwithN}, \eqref{StricartzwithCcubes}, \eqref{pnmsnm}, and the identity \eqref{necessarydelta}, we obtain
\begin{align*}
&\Big|\int_0^T\int_{\mathbb{T}^d }\prod_{j=1}^{2m+1}\widetilde{u}_j\overline{v}\,dxdt\Big|\lesssim \sum_{N_0\sim N_1\geq \cdots \geq N_{2m+1}} \Big\Vert P_{N_0}v\prod_{j=1}^{2m+1}P_{N_j}u_{j}\Big\Vert_{L^1_{t,x}}\\&\lesssim \sum_{N_0\sim N_1\geq \cdots \geq N_{2m+1}} \Big\Vert P_{N_0}v\prod_{j=1}^{m}P_{N_{2j}}u_{2j} \Big\Vert_{L^2_{t,x}}\Big\Vert \prod_{j=0}^{m}P_{N_{2j+1}}u_{2j+1} \Big\Vert_{L^2_{t,x}}\\&\sim \sum_{N_0\sim N_1\geq \cdots \geq N_{2m+1}} \Bigg(\sum_{k_2\in\mathbb{Z}}\Big\Vert P_{C_{k_2}}P_{N_0}v\prod_{j=1}^{m}P_{N_{2j}}u_{2j} \Big\Vert_{L^2_{t,x}}^2\Bigg)^{\frac{1}{2}}\\&\quad\Bigg(\sum_{k_3\in\mathbb{Z}}\Big\Vert P_{C_{k_3}}P_{N_{1}}u_{1}\prod_{j=1}^{m}P_{N_{2j+1}}u_{2j+1}  \Big\Vert_{L^2_{t,x}}^2\Bigg)^{\frac{1}{2}}\\&\lesssim \sum_{N_0\sim N_1\geq \cdots \geq N_{2m+1}}\Big(\sum_{k_2\in\mathbb{Z}}\Vert P_{C_{k_2}}P_{N_0}v\Vert_{L^p_{t,x}}^2\Big)^{\frac{1}{2}}\Vert P_{N_2}u_2\Vert_{L^p_{t,x}}\prod_{j=2}^{m-1}\Vert P_{N_{2j}}u_{2j}\Vert_{L^{p_{d,m}}_{t,x}}\Vert P_{N_{2m}}u_{2m}\Vert_{L^q_{t,x}}\\&\quad \Big(\sum_{k_3\in\mathbb{Z}}\Vert P_{C_{k_3}}P_{N_1}u_1\Vert_{L^p_{t,x}}^2\Big)^{\frac{1}{2}}\Vert P_{N_3}u_3\Vert_{L^p_{t,x}}\prod_{j=2}^{m-1}\Vert P_{N_{2j+1}}u_{2j+1}\Vert_{L^{p_{d,m}}_{t,x}}\Vert P_{N_{2m+1}}u_{2m+1}\Vert_{L^q_{t,x}}\\&\lesssim \sum_{N_0\sim N_1\geq \cdots \geq N_{2m+1}}  N_2^{d-\frac{2(d
    +2)}{p}-s_c (d,m)}N_3^{d-\frac{2(d
    +2)}{p}-s_c (d,m)}N_{2m}^{\frac{d}{2}-\frac{d
    +2}{q}-s_c (d,m)}N_{2m+1}^{\frac{d}{2}-\frac{d
    +2}{q}-s_c (d,m)}\\&\quad\Big(\sum_{k_2\in\mathbb{Z}}\Vert P_{C_{k_2}}P_{N_0}v\Vert_{Y^0}^2\Big)^{\frac{1}{2}} \Big(\sum_{k_3\in\mathbb{Z}}\Vert P_{C_{k_3}}P_{N_1}u_1\Vert_{Y^0}^2\Big)^{\frac{1}{2}}\prod_{j=2}^{2m+1}\Vert P_{N_{j}}u_{j}\Vert_{Y^{s_c (d,m)}} \\&\sim \sum_{N_0\sim N_1\geq \cdots \geq N_{2m+1}} \Big(\frac{N_{2m}}{N_2}\Big)^{\delta}\Big(\frac{N_{2m+1}}{N_3}\Big)^{\delta}\Vert P_{N_0}v\Vert_{Y^0}\Vert P_{N_1}u_1\Vert_{Y^0}\prod_{j=2}^{2m+1}\Vert P_{N_{j}}u_{j}\Vert_{Y^{s_c (d,m)}}\\&\lesssim \sum_{\substack{N_0, N_1\\ N_0\sim N_1}}\Big(\frac{N_{0}}{N_1}\Big)^{s}\Vert P_{N_0}v\Vert_{Y^{-s}}\Vert P_{N_1}u_1\Vert_{Y^s}\sum_{\substack{N_2, N_3\\ N_2\geq N_3}}\Big(\frac{N_{3}} {N_2}\Big)^{\frac{\delta}{2m}}\Vert P_{N_2}u_2\Vert_{Y^{s_c (d,m)}}\Vert P_{N_3}u_3\Vert_{Y^{s_c (d,m)}} \cdots \\&\quad \sum_{\substack{N_{2m}, N_{2m+1}\\ N_{2m}\geq N_{2m+1}}}\Big(\frac{N_{2m+1}}{N_{2m}}\Big)^{\frac{\delta}{2m}}\Vert P_{N_{2m}}u_{2m}\Vert_{Y^{s_c (d,m)}}\Vert P_{N_{2m+1}}u_{2m+1}\Vert_{Y^{s_c (d,m)}} \\&\lesssim \Big(\sum_{N_0}\Vert P_{N_0}v\Vert_{Y^{-s}}^2\Big)^{\frac{1}{2}}\Big(\sum_{N_1}\Vert P_{N_1}u_1\Vert_{Y^{s}}^2\Big)^{\frac{1}{2}}\Big(\sum_{\substack{N_2, N_3\\ N_2\geq N_3}}\Big(\frac{N_{3}} {N_2}\Big)^{\frac{\delta}{2m}}\Vert P_{N_2}u_2\Vert_{Y^{s_c (d,m)}}^2\Big)^{\frac{1}{2}}\\&\quad\Big(\sum_{\substack{N_2, N_3\\ N_2\geq N_3}}\Big(\frac{N_{3}} {N_2}\Big)^{\frac{\delta}{2m}}\Vert P_{N_3}u_3\Vert_{Y^{s_c (d,m)}}^2\Big)^{\frac{1}{2}} \cdots  \Big(\sum_{\substack{N_{2m}, N_{2m+1}\\ N_{2m}\geq N_{2m+1}}}\Big(\frac{N_{2m+1}} {N_{2m}}\Big)^{\frac{\delta}{2m}}\Vert P_{N_{2m}}u_{2m}\Vert_{Y^{s_c (d,m)}}^2\Big)^{\frac{1}{2}}\\&\quad\Big(\sum_{\substack{N_{2m}, N_{2m+1}\\ N_{2m}\geq N_{2m+1}}}\Big(\frac{N_{2m+1}} {N_{2m}}\Big)^{\frac{\delta}{2m}}\Vert P_{N_{2m+1}}u_{2m+1}\Vert_{Y^{s_c (d,m)}}^2\Big)^{\frac{1}{2}}\\&\lesssim \Vert v\Vert_{Y^{-s}}\Vert u_1\Vert_{Y^{s}}\prod_{j=2}^{2m+1}\Vert u_j\Vert_{Y^{s_c (d,m)}}.
\end{align*}
\noindent
\\{\bf Case B.} $N_1\sim N_2\geq N_0\geq N_3\geq  \cdots\geq N_{2m+1}$. 
\vspace{0.13cm}
\\
Similary as in previous case, by cube decomposition, H\"older inequality, Strichartz inequalities \eqref{StricartzwithN}, \eqref{StricartzwithCcubes}, \eqref{pnmsnm} we end up with
\begin{align*}
&\Big|\int_0^T\int_{\mathbb{T}^d }\prod_{j=1}^{2m+1}\widetilde{u}_j\overline{v}\,dxdt\Big|\\&\lesssim \sum_{N_1\sim N_2\geq N_0\geq \cdots \geq N_{2m+1}} \Big\Vert \prod_{j=0}^{m}P_{N_{2j+1}}u_{2j+1} \Big\Vert_{L^2_{t,x}}\Big\Vert P_{N_0}v\prod_{j=1}^{m}P_{N_{2j}}u_{2j} \Big\Vert_{L^2_{t,x}}\\&\lesssim \sum_{N_1\sim N_2\geq N_0\geq \cdots \geq N_{2m+1}} \Big(\frac{N_{2m}}{N_0}\Big)^{\delta}\Big(\frac{N_{2m+1}}{N_3}\Big)^{\delta}N_0^{s+s_c (d,m)}N_1^{-s-s_c (d,m)}\Vert P_{N_{0}}v \Vert_{Y^{-s}}\Vert P_{N_1}u_1\Vert_{Y^{s}}\\&\quad\Vert P_{N_2}u_2\Vert_{Y^{s_c (d,m)}}\prod_{j=3}^{2m+1}\Vert P_{N_{j}}u_{j} \Vert_{Y^{s_c (d,m)}}\\&\lesssim \sum_{\substack{N_1,N_2\\N_1\sim N_2}}\Vert P_{N_1}u_1\Vert_{Y^s}\Vert P_{N_2}u_2\Vert_{Y^{s_c (d,m)}}\sum_{N_0\geq N_3\geq  \cdots \geq N_{2m+1}}\Big(\frac{N_{2m}}{N_0}\Big)^{\delta}\Big(\frac{N_{2m+1}}{N_3}\Big)^{\delta}\\
&\qquad \times  \Vert P_{N_{0}}v \Vert_{Y^{-s}}\prod_{j=3}^{2m+1}\Vert P_{N_{j}}u_{j} \Vert_{Y^{s_c (d,m)}}\\
&\lesssim \Vert v\Vert_{Y^{-s}}\Vert u_1\Vert_{Y^{s}}\prod_{j=2}^{2m+1}\Vert u_j\Vert_{Y^{s_c (d,m)}}.
\end{align*}
\end{proof}
Next, we address the trilinear estimates when $d\geq 5$:
\begin{lemma}\label{lemma2}
 Assume $d\geq 5$, $s\geq s_c (d,1)$, $0<T\leq 1$ and $j_0\in \{1,2,...,d-1\}$. Then we have
\begin{align*}
\Big|\int_0^T\int_{\mathbb{T}^d }\widetilde{u}_1\widetilde{u}_2\widetilde{u}_3\overline{v}\,dxdt\Big|\lesssim \Vert v\Vert_{Y^{-s}}\Vert u_1\Vert_{Y^{s}}\Vert u_2\Vert_{Y^{s_c (d,1)}}\Vert u_3\Vert_{Y^{s_c (d,1)}}
\end{align*}
where $\widetilde{u}_j\in \{u_j, \overline{u}_j\}$.
\end{lemma}
\begin{proof}
First note that for sufficiently small $\varepsilon>0$
 \begin{align*}
     \frac{3(d+2)}{d+1+\varepsilon}>\begin{cases}
2+\frac{8}{d}\quad\,\,\,\,\,\, \text{for}\,\,d\geq 5\,\,\text{even},  \\ 2+\frac{8}{d+1}\quad \text{for}\,\,d\geq 5\,\,\text{odd},
     \end{cases}
 \end{align*}
where the right hand side of above inequality is the same as in \eqref{worstcase} for $d\geq 5$.
\vspace{0.13cm}
\noindent 
\\{\bf Case A.} $N_0\sim N_1\geq N_2\geq N_3$. 
\vspace{0.13cm}
\\
We decompose $\mathbb{Z}^d=\cup_jC_j$ into cubes $C_j\in\mathscr{C}_{N_2}$, and write $C_j\sim C_k$ to indicate that the set $\{k_1+k_2: k_1\in C_j\,\,\text{and}\,\,k_2\in C_k\}$ overlaps the Fourier support of $P_{\leq 2N_2}$. Hence, for a given $C_j$ there are finitely many $C_k$ with $C_j\sim C_k$. For sufficiently small $\varepsilon>0$, by using almost orthogonality, the estimates \eqref{StricartzwithN}, \eqref{StricartzwithCcubes}, and applying Cauchy-Schwarz to sum in $C_j\sim C_k$, we arrive at
\begin{align*}
&\Big|\int_0^T\int_{\mathbb{T}^d }\widetilde{u}_1\widetilde{u}_2\widetilde{u}_3\overline{v}\,dxdt\Big|\\&\lesssim\sum_{N_0\sim N_1\geq N_2\geq N_3}\sum_{C_j\sim C_k}\Vert P_{C_j}P_{N_0}vP_{C_k}P_{N_1} u_1P_{N_2} u_2P_{N_3} u_3\Vert_{L^1_{t,x}} \\&\lesssim\sum_{N_0\sim N_1\geq N_2\geq N_3}\sum_{C_j\sim C_k}\Vert P_{C_j}P_{N_0}v\Vert_{L_{t,x}^{\frac{3(n+2)}{n+1+\varepsilon}}}\Vert P_{C_k}P_{N_1}u_1\Vert_{L_{t,x}^{\frac{3(n+2)}{n+1+\varepsilon}}}\Vert P_{N_2}u_2\Vert_{L_{t,x}^{\frac{3(n+2)}{n+1+\varepsilon}}}\Vert P_{N_3}u_3\Vert_{L_{t,x}^{\frac{n+2}{1-\varepsilon}}}\\&\lesssim\sum_{N_0\sim N_1\geq N_2\geq N_3}N_2^{s_c (d,1)-\varepsilon}N_3^{s_c (d,1)+\varepsilon}\Vert P_{N_2}u_2\Vert_{Y^0}\Vert P_{N_3}u_3\Vert_{Y^0}\sum_{C_j\sim C_k}\Vert P_{C_j}P_{N_0}v\Vert_{Y^0}\Vert P_{C_k}P_{N_1}u_1\Vert_{Y^0} \\&\lesssim \sum_{\substack{N_0,N_1\\N_0\sim N_1}}\Big(\frac{N_0}{N_1}\Big)^s\Vert P_{N_0}v\Vert_{Y^{-s}}\Vert P_{N_1}u_1\Vert_{Y^s}\sum_{\substack{N_2,N_2\\N_2\geq N_3}}\Big(\frac{N_3}{N_2}\Big)^{\varepsilon}\Vert P_{N_2}u_2\Vert_{Y^{s_c (d,1)}}\Vert P_{N_3}u_3\Vert_{Y^{s_c (d,1)}} \\&\lesssim \Big(\sum_{N_0}\Vert P_{N_0}v\Vert_{Y^{-s}}^2\Big)^{\frac{1}{2}}\Big(\sum_{N_1}\Vert P_{N_1}u_1\Vert_{Y^{s}}^2\Big)^{\frac{1}{2}}\Big(\sum_{\substack{N_2,N_3\\N_2\geq N_3}}\Big(\frac{N_3}{N_2}\Big)^{\varepsilon}\Vert P_{N_2}u_2\Vert_{Y^{s_c (d,1)}}^2\Big)^{\frac{1}{2}}\\
&\qquad \times \Big(\sum_{\substack{N_2,N_3\\N_2\geq N_3}}\Big(\frac{N_3}{N_2}\Big)^{\varepsilon}\Vert P_{N_3}u_3\Vert_{Y^{s_c (d,1)}}^2\Big)^{\frac{1}{2}}\\
&\lesssim \Vert v\Vert_{Y^{-s}}\Vert u_1\Vert_{Y^{s}}\Vert u_2\Vert_{Y^{s_c (d,1)}}\Vert u_3\Vert_{Y^{s_c (d,1)}}.
\end{align*}
\noindent
\\{\bf Case B.} $N_1\sim N_2\geq N_0\geq N_3$. \vspace{0.13cm}
\\ We proceed as in the previous case yet, in this case, decomposing $\mathbb{Z}^d=\cup_jC_j$ into cubes $C_j$ of side length $N_0$ to obtain
\begin{align*}
&\Big|\int_0^T\int_{\mathbb{T}^d }\widetilde{u}_1\widetilde{u}_2\widetilde{u}_3\overline{v}\,dxdt\Big| \\&\lesssim\sum_{N_1\sim N_2\geq N_0\geq N_3}\sum_{C_j\sim C_k}\Vert P_{C_j}P_{N_1}u_1\Vert_{L_{t,x}^{\frac{3(n+2)}{n+1+\varepsilon}}}\Vert P_{C_k}P_{N_2}u_2\Vert_{L_{t,x}^{\frac{3(n+2)}{n+1+\varepsilon}}}\Vert P_{N_0}v\Vert_{L_{t,x}^{\frac{3(n+2)}{n+1+\varepsilon}}}\Vert P_{N_3}u_3\Vert_{L_{t,x}^{\frac{n+2}{1-\varepsilon}}}\\&\lesssim\sum_{N_1\sim N_2\geq N_0\geq N_3}N_0^{s_c (d,1)-\varepsilon}N_3^{s_c (d,1)+\varepsilon}\Vert P_{N_0}v\Vert_{Y^0}\Vert P_{N_3}u_3\Vert_{Y^0}\sum_{C_j\sim C_k}\Vert P_{C_j}P_{N_1}u_1\Vert_{Y^0}\Vert P_{C_k}P_{N_2}u_2\Vert_{Y^0}\\&\lesssim\sum_{N_1\sim N_2\geq N_0\geq N_3}N_0^{s+s_c (d,1)-\varepsilon}N_1^{-s-s_c (d,1)}\Vert P_{N_0}v\Vert_{Y^{-s}} \Vert P_{N_1}u_1\Vert_{Y^s}\Vert P_{N_2}u_2\Vert_{Y^{s_c (d,1)}}N_3^{\varepsilon}\Vert P_{N_3}u_3\Vert_{Y^{s_c (d,1)}} \\&\lesssim\sum_{\substack{N_1,N_2\\N_1\sim N_2}}\Vert P_{N_1}u_1\Vert_{Y^{s}}\Vert P_{N_2}u_2\Vert_{Y^{s_c (d,1)}}\sum_{\substack{N_0,N_3\\N_0\geq N_3}}\Big(\frac{N_3}{N_0}\Big)^{\varepsilon}\Vert P_{N_0}v\Vert_{Y^{-s}}\Vert P_{N_3}u_3\Vert_{Y^{s_c (d,1)}}\\&\lesssim \Vert v\Vert_{Y^{-s}}\Vert u_1\Vert_{Y^{s}}\Vert u_2\Vert_{Y^{s_c (d,1)}}\Vert u_3\Vert_{Y^{s_c (d,1)}}.
\end{align*}
\end{proof}
Lastly, we consider the trilinear estimate in four dimensions, which corresponds to the energy critical HNLS \eqref{eq:NLS}.
\begin{lemma}\label{lemma3}
 Let $j_0\in \{1,3\}$, $s\geq s_c (4,1)=1$, and $0<T\leq 1$. Then we have
\begin{align*}
\Big|\int_0^T\int_{\mathbb{T}^4}\widetilde{u}_1\widetilde{u}_2\widetilde{u}_3\overline{v}\,dxdt\Big|\lesssim \Vert v\Vert_{Y^{-s}}\Vert u_1\Vert_{Y^{s}}\Vert u_2\Vert_{Y^{1}}\Vert u_3\Vert_{Y^{1}}
\end{align*}
where $\widetilde{u}_j\in \{u_j, \overline{u}_j\}$.
\end{lemma}
\begin{proof}
First note that in view of Definition \ref{DEF:adm} we are allowed to pick $p>\frac{10}{3}$ to be able to utilize Strichartz estimates of Lemma \ref{LEM:Stri_Y} and Corollary \ref{COR:Stri_GY}.
\vspace{0.13cm}
\noindent
\\{\bf Case A.} $N_0\sim N_1\geq N_2\geq N_3$. \vspace{0.13cm}
\\
By the cube decomposition implementing to the cubes $C_j$ with side length $N_2$, the estimates \eqref{StricartzwithN}, \eqref{StricartzwithCcubes}, and Cauchy-Schwarz for the sum in $C_j\sim C_k$, we obtain
\begin{align*}
&\Big|\int_0^T\int_{\mathbb{T}^d }\widetilde{u}_1\widetilde{u}_2\widetilde{u}_3\overline{v}\,dxdt\Big|\\&\lesssim\sum_{N_0\sim N_1\geq N_2\geq N_3}\sum_{C_j\sim C_k}\Vert P_{C_j}P_{N_0}vP_{C_k}P_{N_1} u_1P_{N_2} u_2P_{N_3} u_3\Vert_{L^1_{t,x}} \\&\lesssim\sum_{N_0\sim N_1\geq N_2\geq N_3}\sum_{C_j\sim C_k}\Vert P_{C_j}P_{N_0}v\Vert_{L_{t,x}^{\frac{7}{2}}}\Vert P_{C_k}P_{N_1}u_1\Vert_{L_{t,x}^{\frac{7}{2}}}\Vert P_{N_2}u_2\Vert_{L_{t,x}^{\frac{7}{2}}}\Vert P_{N_3}u_3\Vert_{L_{t,x}^{7}}\\&\lesssim\sum_{N_0\sim N_1\geq N_2\geq N_3}N_2^{\frac{6}{7}}N_3^{\frac{8}{7}}\Vert P_{N_2}u_2\Vert_{Y^0}\Vert P_{N_3}u_3\Vert_{Y^0}\sum_{C_j\sim C_k}\Vert P_{C_j}P_{N_0}v\Vert_{Y^0}\Vert P_{C_k}P_{N_1}u_1\Vert_{Y^0} \\&\lesssim \sum_{\substack{N_0,N_1\\N_0\sim N_1}}\Big(\frac{N_0}{N_1}\Big)^s\Vert P_{N_0}v\Vert_{Y^{-s}}\Vert P_{N_1}u_1\Vert_{Y^s}\sum_{\substack{N_2,N_2\\N_2\geq N_3}}\Big(\frac{N_3}{N_2}\Big)^{\frac{1}{7}}\Vert P_{N_2}u_2\Vert_{Y^{1}}\Vert P_{N_3}u_3\Vert_{Y^{1}} \\&\lesssim \Big(\sum_{N_0}\Vert P_{N_0}v\Vert_{Y^{-s}}^2\Big)^{\frac{1}{2}}\Big(\sum_{N_1}\Vert P_{N_1}u_1\Vert_{Y^{s}}^2\Big)^{\frac{1}{2}}\Big(\sum_{\substack{N_2,N_3\\N_2\geq N_3}}\Big(\frac{N_3}{N_2}\Big)^{\frac{1}{7}}\Vert P_{N_2}u_2\Vert_{Y^{1}}^2\Big)^{\frac{1}{2}}\\&\quad\Big(\sum_{\substack{N_2,N_3\\N_2\geq N_3}}\Big(\frac{N_3}{N_2}\Big)^{\frac{1}{7}}\Vert P_{N_3}u_3\Vert_{Y^{1}}^2\Big)^{\frac{1}{2}}\\&\lesssim \Vert v\Vert_{Y^{-s}}\Vert u_1\Vert_{Y^{s}}\Vert u_2\Vert_{Y^{1}}\Vert u_3\Vert_{Y^{1}}.
\end{align*}
\noindent
{\bf Case B.} $N_1\sim N_2\geq N_0\geq N_3$. \vspace{0.13cm}
\\
Using cubes $C_j$ with side length $N_0$ this time in the cube decomposition and proceeding as before, we get
\begin{align*}
&\Big|\int_0^T\int_{\mathbb{T}^d }\widetilde{u}_1\widetilde{u}_2\widetilde{u}_3\overline{v}\,dxdt\Big| \\&\lesssim\sum_{N_1\sim N_2\geq N_0\geq N_3}\sum_{C_j\sim C_k}\Vert P_{C_j}P_{N_1}u_1\Vert_{L_{t,x}^{\frac{7}{2}}}\Vert P_{C_k}P_{N_2}u_2\Vert_{L_{t,x}^{\frac{7}{2}}}\Vert P_{N_0}v\Vert_{L_{t,x}^{\frac{7}{2}}}\Vert P_{N_3}u_3\Vert_{L_{t,x}^{7}}\\&\lesssim\sum_{N_1\sim N_2\geq N_0\geq N_3}N_0^{\frac{6}{7}}N_3^{\frac{8}{7}}\Vert P_{N_0}v\Vert_{Y^0}\Vert P_{N_3}u_3\Vert_{Y^0}\sum_{C_j\sim C_k}\Vert P_{C_j}P_{N_1}u_1\Vert_{Y^0}\Vert P_{C_k}P_{N_2}u_2\Vert_{Y^0}\\&\lesssim\sum_{N_1\sim N_2\geq N_0\geq N_3}N_0^{s+\frac{6}{7}}N_1^{-s-1}\Vert P_{N_0}v\Vert_{Y^{-s}} \Vert P_{N_1}u_1\Vert_{Y^s}\Vert P_{N_2}u_2\Vert_{Y^{1}}N_3^{\frac{1}{7}}\Vert P_{N_3}u_3\Vert_{Y^{1}} \\&\lesssim\sum_{\substack{N_1,N_2\\N_1\sim N_2}}\Vert P_{N_1}u_1\Vert_{Y^{s}}\Vert P_{N_2}u_2\Vert_{Y^{1}}\sum_{\substack{N_0,N_3\\N_0\geq N_3}}\Big(\frac{N_3}{N_0}\Big)^{\frac{1}{7}}\Vert P_{N_0}v\Vert_{Y^{-s}}\Vert P_{N_3}u_3\Vert_{Y^{1}}\\&\lesssim \Vert v\Vert_{Y^{-s}}\Vert u_1\Vert_{Y^{s}}\Vert u_2\Vert_{Y^{1}}\Vert u_3\Vert_{Y^{1}}.
\end{align*}
\end{proof}
\begin{remark}
The term $u_1$ appearing on the right side of the estimates in Lemmas \ref{lemma1}, \ref{lemma2}, and \ref{lemma3},  can be exchanged with any other $u_j$ for $j\in\{2,...,2m+1\}$.
\end{remark}

\section{Well-posedness for HNLS}\label{wellposednessHNLS}
In this section, we outline the proof of Theorem~\ref{THM:main} using the multilinear estimates established in Section~\ref{SectNonlinear}.  
For simplicity, we focus on the scaling-critical regularities, relying on Lemmas~\ref{lemma1}, Lemmas~\ref{lemma2}, and Lemmas~\ref{lemma3}.  
The same argument also applies in the subcritical cases, where one instead makes use of Lemma~\ref{PROP:T2m}.

To begin with, we have the following result.
\begin{proposition}\label{keyproposition}
 Let $s\geq s_c(d,m)$ be fixed where
 \begin{equation*}
 \begin{aligned}
d&=2, \quad\,\,\,\,\, m\geq4,\quad j_0=1,\\ d&=3, 4,\quad m\geq2,\quad j_0\in\{1,...,d-1\},\\ d&=4, \quad\,\,\,\,\, m=1,\quad j_0=1,3,\\ d&\geq 5,\,\,\,\,\, \quad m\geq 1,\quad j_0\in\{1,...,d-1\}. 
  \end{aligned}
\end{equation*}
 Then, for all $0<T\leq 1$, and $u_j\in X^s([0,T))$, $j=1,...,2m+1$, we have
 \begin{align}\label{mainest}
 \Big\Vert \int_0^te^{i(t-\tau)\Delta_{\pm}} \prod_{j=1}^{2m+1}\widetilde{u}_j\,d\tau\Big\Vert_{X^s([0,T))}\lesssim \sum_{l=1}^{2m+1}\Vert u_l\Vert_{X^s([0,T))}\prod_{\substack{j=1\\j\neq l}}^{2m+1}\Vert u_j\Vert_{X^{s_c(d,m)}([0,T))}
 \end{align}
 where $\widetilde{u}_j\in \{u_j, \overline{u}_j\}$.
\end{proposition}
\begin{proof}
    The estimate \eqref{mainest} directly follows from combining Lemma \ref{lemmaduality}, and Lemmas \ref{lemma1}, \ref{lemma2}, and \ref{lemma3}.
\end{proof}
\begin{proof}[Proof of Theorem \ref{THM:main}]
We give the proof only for $s=s_c(d,m)$ and small data, for details of the large data case and uniqueness, see \cite{HTT11,KV14}. Let $\delta>0$, $\varepsilon>0$ to be determined later. We apply contraction argument to the operator
\begin{align}\label{Gammaoperator}
\Gamma(u)(t):=e^{it\Delta_{\pm}}u_0\pm i\int_0^te^{i(t-\tau)\Delta_{\pm}}(|u|^{2m}u)(\tau)\,d\tau
\end{align}
on the ball
\begin{align*}
    B_{\delta}:=\{u\in X^{s_c(d,m)}([0,1))\cap C([0,1);H^{s_c(d,m)}(\mathbb{T}^d )): \Vert u\Vert_{X^{s_c(d,m)}([0,1))}\leq \delta\},
\end{align*}
for $u_0\in H^{s_c(d,m)}(\mathbb{T}^d )$ satisfying $\Vert u_0\Vert_{H^{s_c(d,m)}(\mathbb{T}^d )}\leq \varepsilon$. Therefore, for $u\in B_{\delta}$, applying Lemma \ref{linearestlemma} and Proposition \ref{keyproposition} to the operator \eqref{Gammaoperator} leads to 
\begin{align*}
    \Vert \Gamma(u)\Vert_{X^{s_c(d,m)}([0,1))}&\leq \Vert e^{it\Delta_{\pm}}u_0\Vert_{X^{s_c(d,m)}([0,1))}+ \Big\Vert \int_0^te^{i(t-\tau)\Delta_{\pm}}(|u|^{2m}u)(\tau)\,d\tau\Big\Vert_{X^{s_c(d,m)}([0,1))}\\&\leq \Vert u_0\Vert_{H^{s_c(d,m)}(\mathbb{T}^d )}+C\Vert u\Vert_{X^{s_c(d,m)}([0,1))}^{2m+1}\\&\leq \varepsilon+C\delta^{2m+1}\leq \delta
\end{align*}
provided that we pick $\varepsilon=\frac{\delta}{2}$ and $\delta=(2C)^{-\frac{1}{2m}}$. This shows that $\Gamma$ maps $B_{\delta}$ to itself. Continuing in the same way by selecting $\delta$ same as above, we obtain
\begin{align*}
    \Vert \Gamma(u)-\Gamma(v)\Vert_{X^{s_c(d,m)}([0,1))}&\leq C(\Vert u\Vert_{X^{s_c(d,m)}([0,1))}^{2m}+\Vert v\Vert_{X^{s_c(d,m)}([0,1))}^{2m})\Vert u-v\Vert_{X^{s_c(d,m)}([0,1))}\\&\leq \frac{1}{2}\Vert u-v\Vert_{X^{s_c(d,m)}([0,1))}
\end{align*}
which proves that $\Gamma$ is a contraction so that it has a unique fixed point.
\end{proof}

\appendix

\section{On the cubic NLS}\label{appendix}

In this appendix, 
we consider the cubic Schrodinger equation on $\T^3$ 
\begin{equation}
\label{eq:NLScubic}
\begin{cases}
i\pa_t u + \Delta u \pm |u|^{2}u=0,\\
u(t,x)|_{t=0}=u_0(x),
\end{cases}
\quad (t,x)\in\mathbb{R}\times\mathbb{T}^{3},
\end{equation}
The well-posedness of \eqref{eq:NLScubic} in the critical space $H^{\frac12}(\T^3)$ was claimed in Theorem \ref{THM:elliptic},
however, its proof was left open in \cite{YW13}. 
The proof follows the same strategy as in other cases treated in \cite{YW13}, 
with the additional use of Strichartz estimates from \cite{BD15,KV14}. 
For completeness, we provide the proof in this appendix.

Using the scale-invariant Strichartz estimates of \cite{B93}, the well-posedness of \eqref{eq:NLScubic} at the scaling-critical regularity indices for all cases in Theorem~\ref{THM:elliptic}, except when $d=3,4$ with $m=1$, was established in  \cite{HTT11,YW13}.  
Subsequently, for the case $d=4$, $m=1$, well-posedness in the critical space was resolved in \cite{KV14}, building on the results of \cite{BD15}.  
Thus, the case $(d,m)=(3,1)$ remains the only cases from Theorem \ref{THM:elliptic}.

\begin{theorem}\label{thmwellposedfornls}
Fix $s\geq \frac{1}{2}$ and let $u_0\in H^{s}(\mathbb{T}^3)$. Then, there exists a time $T=T(u_0)$ and a unique solution $$u\in C([0,T);H^{s}(\mathbb{T}^3))\cap X^{s}([0,T))$$ to  \eqref{eq:NLScubic} with $m=1$. Moreover, \eqref{eq:NLScubic} is globally well-posed for sufficiently small initial data $u_0\in H^{1}(\mathbb{T}^3)$.
\end{theorem}
We use the following Strichartz estimates associated with Schr\"odinger equation \ref{eq:NLScubic} to conclude Theorem \ref{thmwellposedfornls}.   
\begin{lemma}[\cite{BD15, KV14}]\label{Strichartzestfornls} For all $N\geq 1$ and $p>\frac{2(d+2)}{d}$, we have
\begin{equation*}
\begin{aligned}
  \Vert P_{\leq N}u\Vert_{L^p([0,T)\times \mathbb{T}^d )} \lesssim N^{\frac{d}{2}-\frac{d+2}{p}}\Vert P_{\leq N}u\Vert_{Y^0([0,T))}. 
\end{aligned}
\end{equation*}
Moreover, for all $C\in \mathscr{C}_N$,
\begin{equation*}
\begin{aligned}
  \Vert P_C u\Vert_{L^p([0,T)\times \mathbb{T}^d )} \lesssim N^{\frac{d}{2}-\frac{d+2}{p}}\Vert P_C u\Vert_{Y^0([0,T))}. 
\end{aligned}
\end{equation*}
\end{lemma}

Then, Theorem \ref{thmwellposedfornls} reduced to the following multilinear estimate.
\begin{lemma}\label{LEM:multi3d}
 Let $s\geq \frac{1}{2}$, and $0<T\leq 1$. Then we have
\begin{align*}
\Big|\int_0^T\int_{\mathbb{T}^3}\widetilde{u}_1\widetilde{u}_2\widetilde{u}_3\overline{v}\,dxdt\Big|\lesssim \Vert v\Vert_{Y^{-s}}\Vert u_1\Vert_{Y^{s}}\Vert u_2\Vert_{Y^{\frac{1}{2}}}\Vert u_3\Vert_{Y^{\frac{1}{2}}}
\end{align*}
where $\widetilde{u}_j\in \{u_j, \overline{u}_j\}$.
\end{lemma}
\begin{proof}
It suffices to consider the following regions. \noindent \vspace{0.13cm}
\\{\bf Case A.} $N_0\sim N_1\geq N_2\geq N_3$. \vspace{0.13cm}
\\ Let $P_{C_j}$ denote the Fourier projection onto cube of size $N_2$. As above, we write $C_j\sim C_k$ whenever the sum set overlaps the Fourier support of $P_{\leq 2N_2}$. Therefore, by H\"older inequality, Lemma \ref{Strichartzestfornls}, and Cauchy-Schwarz in sum $C_j\sim C_k$, we have    
\begin{align*}
&\Big|\int_0^T\int_{\mathbb{T}^d }\widetilde{u}_1\widetilde{u}_2\widetilde{u}_3\overline{v}\,dxdt\Big|\\&\lesssim\sum_{N_0\sim N_1\geq N_2\geq N_3}\sum_{C_j\sim C_k}\Vert P_{C_j}P_{N_0}vP_{C_k}P_{N_1} u_1P_{N_2} u_2P_{N_3} u_3\Vert_{L^1_{t,x}} \\&\lesssim\sum_{N_0\sim N_1\geq N_2\geq N_3}\sum_{C_j\sim C_k}\Vert P_{C_j}P_{N_0}v\Vert_{L_{t,x}^{\frac{7}{2}}}\Vert P_{C_k}P_{N_1}u_1\Vert_{L_{t,x}^{\frac{7}{2}}}\Vert P_{N_2}u_2\Vert_{L_{t,x}^{\frac{7}{2}}}\Vert P_{N_3}u_3\Vert_{L_{t,x}^{7}}\\&\lesssim\sum_{N_0\sim N_1\geq N_2\geq N_3}N_2^{\frac{3}{14}}N_3^{\frac{11}{14}}\Vert P_{N_2}u_2\Vert_{Y^0}\Vert P_{N_3}u_3\Vert_{Y^0}\sum_{C_j\sim C_k}\Vert P_{C_j}P_{N_0}v\Vert_{Y^0}\Vert P_{C_k}P_{N_1}u_1\Vert_{Y^0} \\&\lesssim \sum_{\substack{N_0,N_1\\N_0\sim N_1}}\Big(\frac{N_0}{N_1}\Big)^s\Vert P_{N_0}v\Vert_{Y^{-s}}\Vert P_{N_1}u_1\Vert_{Y^s}\sum_{\substack{N_2,N_2\\N_2\geq N_3}}\Big(\frac{N_3}{N_2}\Big)^{\frac{2}{7}}\Vert P_{N_2}u_2\Vert_{Y^{\frac{1}{2}}}\Vert P_{N_3}u_3\Vert_{Y^{\frac{1}{2}}} \\&\lesssim \Big(\sum_{N_0}\Vert P_{N_0}v\Vert_{Y^{-s}}^2\Big)^{\frac{1}{2}}\Big(\sum_{N_1}\Vert P_{N_1}u_1\Vert_{Y^{s}}^2\Big)^{\frac{1}{2}}\Big(\sum_{\substack{N_2,N_3\\N_2\geq N_3}}\Big(\frac{N_3}{N_2}\Big)^{\frac{2}{7}}\Vert P_{N_2}u_2\Vert_{Y^{\frac{1}{2}}}^2\Big)^{\frac{1}{2}}\\&\quad\Big(\sum_{\substack{N_2,N_3\\N_2\geq N_3}}\Big(\frac{N_3}{N_2}\Big)^{\frac{2}{7}}\Vert P_{N_3}u_3\Vert_{Y^{\frac{1}{2}}}^2\Big)^{\frac{1}{2}}\\&\lesssim \Vert v\Vert_{Y^{-s}}\Vert u_1\Vert_{Y^{s}}\Vert u_2\Vert_{Y^{\frac{1}{2}}}\Vert u_3\Vert_{Y^{\frac{1}{2}}}.
\end{align*}
\noindent
\\{\bf Case B.} $N_1\sim N_2\geq N_0\geq N_3$. \vspace{0.13cm} \\
We continue similarly as above to get
\begin{align*}
&\Big|\int_0^T\int_{\mathbb{T}^d }\widetilde{u}_1\widetilde{u}_2\widetilde{u}_3\overline{v}\,dxdt\Big| \\&\lesssim\sum_{N_1\sim N_2\geq N_0\geq N_3}\sum_{C_j\sim C_k}\Vert P_{C_j}P_{N_1}u_1\Vert_{L_{t,x}^{\frac{7}{2}}}\Vert P_{C_k}P_{N_2}u_2\Vert_{L_{t,x}^{\frac{7}{2}}}\Vert P_{N_0}v\Vert_{L_{t,x}^{\frac{7}{2}}}\Vert P_{N_3}u_3\Vert_{L_{t,x}^{7}}\\&\lesssim\sum_{N_1\sim N_2\geq N_0\geq N_3}N_0^{\frac{3}{14}}N_3^{\frac{11}{14}}\Vert P_{N_0}v\Vert_{Y^0}\Vert P_{N_3}u_3\Vert_{Y^0}\sum_{C_j\sim C_k}\Vert P_{C_j}P_{N_1}u_1\Vert_{Y^0}\Vert P_{C_k}P_{N_2}u_2\Vert_{Y^0}\\&\lesssim\sum_{N_1\sim N_2\geq N_0\geq N_3}N_0^{s+\frac{3}{14}}N_1^{-s-\frac{1}{2}}\Vert P_{N_0}v\Vert_{Y^{-s}} \Vert P_{N_1}u_1\Vert_{Y^s}\Vert P_{N_2}u_2\Vert_{Y^{\frac{1}{2}}}N_3^{\frac{2}{7}}\Vert P_{N_3}u_3\Vert_{Y^{\frac{1}{2}}} \\&\lesssim\sum_{\substack{N_1,N_2\\N_1\sim N_2}}\Vert P_{N_1}u_1\Vert_{Y^{s}}\Vert P_{N_2}u_2\Vert_{Y^{\frac{1}{2}}}\sum_{\substack{N_0,N_3\\N_0\geq N_3}}\Big(\frac{N_3}{N_0}\Big)^{\frac{2}{7}}\Vert P_{N_0}v\Vert_{Y^{-s}}\Vert P_{N_3}u_3\Vert_{Y^{\frac{1}{2}}}\\&\lesssim \Vert v\Vert_{Y^{-s}}\Vert u_1\Vert_{Y^{s}}\Vert u_2\Vert_{Y^{\frac{1}{2}}}\Vert u_3\Vert_{Y^{\frac{1}{2}}}.
\end{align*}
Thus, we finish the proof.
\end{proof}

The proof of Theorem~\ref{thmwellposedfornls} then follows from Lemma~\ref{LEM:multi3d} by a standard argument.  
As it is analogous to the proof of Theorem~\ref{THM:main} above, we omit the details.

\end{document}